\documentclass[reqno,a4paper,11pt]{amsproc}
\usepackage{amsmath,amscd,amsfonts,amssymb}
\usepackage{mathrsfs,dsfont}

\newcommand\N{\mathbb{N}}
\newcommand\R{\mathbb{R}}
\newcommand\C{\mathbb{C}}
\newcommand\Z{\mathbb{Z}}
\newcommand\T{\mathbb{T}}
\newcommand\1{\mathds{1}}
\newcommand\W{\mathscr{W}}
\newcommand\Wa{\mathcal{W}}
\newcommand\eps{\varepsilon}
\newcommand\X{\mathscr{X}}

\newcommand\F{\mathcal{F}}

\newcommand\A{\mathcal{A}}

\newcommand{\proj}{\operatorname{Proj}}
\newcommand{\E}{\mathcal{E}}

\renewcommand{\P}{\mathbb{P}}

\newcommand\e{\operatorname{\mathbb{E}}}
\newcommand\p{\operatorname{\mathbb{P}}}

\newcommand\var{\operatorname{Var}}

\renewcommand\leq{\leqslant}
\renewcommand\geq{\geqslant}
\renewcommand\le{\leqslant}
\renewcommand\ge{\geqslant}

\renewcommand\Re{\operatorname{Re}}
\renewcommand\Im{\operatorname{Im}}

\newcommand\ft[1]{\widehat #1}
\newcommand\dotprod[2]{\langle #1 , #2 \rangle}
\newcommand\mes{\operatorname{mes}}

\theoremstyle{plain}
\newtheorem{theorem-main}{Theorem}
\newtheorem{corollary-main}{Corollary}
\newtheorem{theorem}{Theorem}[section]
\newtheorem{proposition}[theorem]{Proposition}
\newtheorem{lemma}[theorem]{Lemma}
\newtheorem{corollary}[theorem]{Corollary}
\newtheorem*{theorem-a}{Theorem A}
\newtheorem*{theorem-b}{Theorem B}

\theoremstyle{definition}
\newtheorem{definition}{Definition}
\newtheorem*{definition*}{Definition}
\newtheorem*{remark*}{Remark}

\newenvironment{enumerate-math}
{\begin{enumerate}
\addtolength{\itemsep}{5pt}
}
{\end{enumerate}}

\newcommand{\D}{\mathcal{D}}

\newcommand{\rank}{\operatorname{rank}}
\newcommand{\AC}{A}
\newcommand{\AP}{\mathcal{AP}}

\newcommand{\open}{\mathcal{O}}

\begin{document}

\title{Separating signal from noise}

\author{Nir Lev}
\thanks{Research of N.L. is partially supported by the Israel Science Foundation grant No.
225/13.}
\address{Bar-Ilan University, Ramat-Gan 52900, Israel.}
\email{levnir@math.biu.ac.il}

\author{Ron Peled}
\thanks{Research of R.P. is partially supported by an ISF grant and an IRG
grant.}
\address{School of Mathematical sciences, Tel-Aviv University, Tel-Aviv 69978, Israel.}
\email{peledron@post.tau.ac.il}
\urladdr{http://www.math.tau.ac.il/~peledron}

\author{Yuval Peres}
\address{Microsoft Research, One Microsoft Way, Redmond, WA 98052, USA.}
\email{peres@microsoft.com}

\begin{abstract}
Suppose that a sequence of numbers $x_n$ (a `signal') is transmitted
through a noisy channel. The receiver observes a noisy version of
the signal with additive random fluctuations, $x_n + \xi_n$, where
$\xi_n$ is a sequence of independent standard Gaussian random
variables. Suppose further that the signal is known to come from
some fixed space $\X$ of possible signals. Is it possible to fully
recover the transmitted signal from its noisy version? Is it
possible to at least detect that a non-zero signal was transmitted?

In this paper we consider the case in which signals are infinite
sequences and the recovery or detection are required to hold with
probability one. We provide conditions on the space $\X$ for
checking whether detection or recovery are possible. We also analyze
in detail several examples including spaces of Fourier transforms of
measures, spaces with fixed amplitudes and the space of almost
periodic functions. Many of our examples exhibit critical phenomena,
in which a sharp transition is made from a regime in which recovery
is possible to a regime in which even detection is impossible.
\end{abstract}

\maketitle

\renewcommand{\thefootnote}{\fnsymbol{footnote}}
\renewcommand{\thefootnote}{\arabic{footnote}}


\section{Introduction}
\label{section:introduction}

\subsection{} Let $x = \{x_n\}$ be a sequence of numbers, which we
consider as a ``signal''. Suppose that the signal is transmitted
through a noisy channel, and the receiver observes the signal with
additive random fluctuations, namely the sequence $\{x_n + \xi_n\}$
where the $\xi_n$ are independent standard Gaussian random
variables. When is it possible for the receiver to recover the
original signal from its noisy version?

This type of question has been considered by many authors in
different versions, see
Section~\ref{section:definitions_related_works}. In this work we
shall consider a setting which seems to have received little
attention. We consider infinite signals, and ask for \emph{perfect
recovery} of the transmitted signal by the receiver. Clearly, for
perfect recovery to be possible, the receiver must have some prior
information about the transmitted signal. This is imposed by
requiring that the signal $x$ belongs to a given class of sequences
$\X$, which is known to the receiver.

In addition, there are several possible interpretations to the
notion of ``recovery'' in the random setting. In this paper we focus
on \emph{almost sure recovery}, that is, we require that the
receiver may recover every signal from $\X$ with probability one.
The following definition formalizes these ideas.

\begin{definition}
We say that the space \emph{$\X\subset \R^\N$ admits recovery} if
there exists a Borel measurable mapping $T:\R^\N\to\R^\N$, such that
for each $x \in \X$ we have $T(\{x_n + \xi_n\}) = \{x_n\}$ almost
surely.
\end{definition}

One may also consider a variant of this problem, in which the
receiver is asked merely to detect the existence of a signal in the
observed sequence. In other words, one is required only to
distinguish noisy versions of signals from ``pure noise''. Again,
our setting differs from most of the existing literature in that we
take our signals to be infinite sequences, and ask for almost sure
detection.

\begin{definition}
We say that the space \emph{$\X\subset\R^\N$ admits detection} if
there exists a Borel measurable mapping $T:\R^\N\to\{0,1\}$ such
that:
\begin{enumerate-math}
\item For every $x\in \X$, we have $T(\{x_n+\xi_n\})=1$ almost surely.
\item $T(\{\xi_n\})=0$ almost surely.
\end{enumerate-math}
\end{definition}

Similar definitions apply for any countable index set and for
complex-valued signals. We discuss these definitions in more detail
in Section~\ref{section:definitions_related_works} below.

It is natural to expect that for $\X$ to admit recovery, the signals
in $\X$ should be, in a sense, separated from each other. Indeed, a
basic result is that for $\X$ to admit recovery, we must have
\begin{equation}\label{eq:recover_nec_cond}
  \sum_n |x_n - y_n|^2 = \infty\;\;\text{for every distinct
  $x,y\in\X$}.
\end{equation}
This may be deduced, for example, from Kakutani's theorem on
singularity of product measures \cite{kakutani}. Similarly, a
necessary condition for $\X$ to admit detection is
\begin{equation}\label{eq:detection_nec_cond}
  \sum_n |x_n|^2 = \infty\;\;\text{for every $x\in\X$}.
\end{equation}
These conditions turn out to also be sufficient if $\X$ is a
countable space, see Section~\ref{section:preliminaries}. However,
as we will see below, these conditions are not sufficient in
general.

In this paper we give both necessary and sufficient conditions for a
general space $\X$ to admit recovery or detection. These conditions
are then used to study several examples. Most of these examples
exhibit critical phenomena, in which a sharp transition is made from
a regime in which recovery (or detection) is possible to a regime in
which it is not.

\subsection{} A simple example of an uncountable space which admits
recovery is the space of all periodic signals. Our first substantial
example shows that recovery may be possible from much larger spaces.
Consider a signal $\{x_n\}$, $n\in\Z$, which is the Fourier
transform of a measure $\mu$ on the circle $\T = \R/\Z$. The support
of $\mu$ represents the spectrum of frequencies contained in the
signal. Periodic signals thus correspond to atomic measures
supported by the roots of unity.

We study spaces of signals obtained by restricting the support of
the measure, and show that a sharp transition occurs when the
Hausdorff dimension of the support crosses $1/2$. We define the
dimension of a measure $\mu$ as
\begin{equation}\label{eq:dimension_of_mu}
\dim(\mu):=\min\{\dim(E)\colon E\text{ Borel}, |\mu|(\T\setminus
E)=0\},
\end{equation}
where $\dim(E)$ is the Hausdorff dimension of $E$.

\begin{theorem}\label{thm:Fourier_transforms_of_measures}
  Let $0\le \alpha\le 1$ and let $\X$ be the space of Fourier transforms of all finite, complex measures
  whose dimension is at most $\alpha$.
  \begin{enumerate-math}
    \item If $\alpha<1/2$ then $\X$ admits recovery.
    \item If $\alpha>1/2$ then $\X\setminus\{0\}$ does not admit detection.
  \end{enumerate-math}
\end{theorem}
This theorem is proved and discussed further in
Section~\ref{sec:hausdorff}, see
Theorem~\ref{thm:reconstruction-small-dimension}. We also prove
there a related result, that in the space of all measures supported
on a given, fixed set, recovery is possible if and only if this set
has Lebesgue measure zero.

\subsection{} In Section~\ref{sec:unknown_phase} we study a different
type of signal space. Here, the situation is that the amplitudes of
the transmitted signal are known to the receiver beforehand, and
only the signs, or phases in the complex case, remain unknown. Thus,
the amplitudes are given by a sequence $\{\sigma_n\}$, and the space
consists of signals $\{x_n\}$ satisfying $|x_n| = \sigma_n$. Let us
present two examples of this type, which exhibit rather different
behavior.

The first space, which we term the ``Rademacher space'', consists of
real signals of the form $x_n = \pm \sigma_n$ with all possible
choices of signs allowed. Certainly, one cannot expect this space to
admit recovery since each sign affects only one coordinate (and
hence the necessary condition \eqref{eq:recover_nec_cond} is
violated). However, the detection problem still makes sense.
Condition~\eqref{eq:detection_nec_cond} implies that $\sum
\sigma_n^2 = \infty$ is a necessary condition for detection, but it
turns out that this condition is not sufficient. The following
result is true (see Theorem~\ref{thm:detection_phases}(i)).
\begin{theorem}\label{thm:Rademacher_detection}
  The Rademacher space admits detection if and only if $\sum
  \sigma_n^4 = \infty$.
\end{theorem}

We also consider different examples, in which the set of possible
phases for the signals is restricted in some way. A representative
space is the, so termed, ``trigonometric space''. Here we take
$\{\sigma_n\}$, $n\in\Z$, which we allow also to take complex
values. The signals are indexed by a parameter $t\in\T$ and to each
$t$ corresponds the signal $x_n(t) = e^{2\pi i n t} \sigma_n$. In
this space the phases are highly correlated and both the detection
and recovery problems make sense. One motivation for considering
such a space is by noting that if the $\{\sigma_n\}$ are the Fourier
coefficients of some object on the circle (say, a function or a
measure), then the trigonometric space consists of the Fourier
coefficients for all rotations of this object. Thus, the receiver
seeks to recover the unknown rotation from the noisy signal.
\begin{theorem}\label{thm:trig_space_recovery_intro}
  Suppose the absolute values $\{|\sigma_n|\}$ are non-increasing. There exist absolute constants $0<a<b<\infty$ such that, if
  \begin{equation}\label{eq:cond_non_detect_trig}
    \limsup_{k\to\infty} \frac1{\log k} \sum_{|n|<k} |\sigma_n|^2
  \end{equation}
  is larger than $b$ then recovery is possible in the trigonometric
  space, while if this limit is smaller than $a$
  then detection is impossible.
\end{theorem}
This is proved in
Sections~\ref{sec:Walsh_and_trigonometric_non_detection} and
\ref{sec:Walsh_and_trigonometric_recovery}. The non-increasing
condition should be understood in the sense that $|\sigma_{n+1}|\le
|\sigma_{n}|$ for $n\ge 0$, and $|\sigma_{n-1}|\le |\sigma_n|$ for
$n\le 0$. It is interesting to note that, again, the necessary
condition $\sum \sigma_n^2 = \infty$ is insufficient. However, the
sufficient condition is ``closer'' to this necessary condition than
in the case of the Rademacher space, which is an indication of the
fact that the trigonometric space is much more restricted.

Section~\ref{sec:unknown_phase} contains also other examples in the
same spirit.

\subsection{} The results presented so far consist of analysis of concrete
spaces. We are also interested in establishing useful conditions for
detection and recovery from general spaces of signals. Such
conditions are presented in Section~\ref{section:preliminaries}.

The simplest result relating to detection from general spaces relies
on the notion of volume growth of the space. Assume that the space
$\X$ is such that each coordinate $x_n$ $(n=1,2,3,\ldots)$ may take
only finitely many values. Define
\begin{equation*}
  \proj_k(\X):=\{(x_1,\ldots, x_k)\,:\, x\in\X\},
\end{equation*}
which is then a finite set.
\begin{theorem}\label{thm:general_detection_intro}
Let $\X$ be a Borel subset of $\R^\N$ or $\C^\N$, such that $\sum
|x_n|^2 = \infty$ for all signals in $\X$. If
\begin{equation}\label{eq:discrete_entropy_cond_intro}
  \liminf_{k\to\infty} \frac{\log|\proj_k(\X)|}{\inf_{x\in\X}\sum_{n=1}^k |x_n|^2} <
  \frac{1}{8}
\end{equation}
then $\X$ admits detection.
\end{theorem}

See Corollary~\ref{cor:vol_growth_detect}(i). This theorem is sharp
in the sense that the constant on the right-hand side of
\eqref{eq:discrete_entropy_cond_intro} may not be replaced by an
arbitrarily large constant. A similar theorem is proved for the
recovery problem. We also obtain results which may be used for
spaces in which the coordinates $x_n$ take infinitely many values.
The results yield, for example, the recovery criterion for the
trigonometric space presented above.

We also obtain a condition necessary for detection. The condition
says, roughly, that a space $\X$ does not admit detection if there
exists a probability measure on $\X$ under which two independent
samples are nearly orthogonal. In this sense, a space which admits
detection cannot be ``large in many different directions''.

\begin{theorem}\label{thm:general_non_detection_intro}
Let $\X$ be a Borel subset of $\R^\N$. Suppose that there is a
probability measure $P$ on $\X$ such that
\begin{equation*}
  \liminf_{k \to\infty} \; \e \, \exp \bigg\{\sum_{n=1}^k x_n y_n \bigg\} <
  \infty,
\end{equation*}
where $\{x_n\}$ and $\{y_n\}$ are sampled independently from $P$.
Then $\X$ does not admit detection.
\end{theorem}
This theorem (see Theorem~\ref{thm:general_detection}) is the main
tool which we use to prove non-detection results. In particular, it
is used to obtain the non-detection results presented earlier in the
introduction. The restriction that $\X$ is a subset of $\R^\N$ can
be made without loss of generality. See
Section~\ref{sec:general_conditions} for more details and for a
related theorem giving a condition both necessary and sufficient for
detection.

\subsection{} If a space $\X$ does not admit recovery or detection, it is natural to ask whether this situation could be improved
by sufficiently reducing the noise level. Thus one may ask whether
recovery becomes possible when the noise added to each coordinate
has a sufficiently small variance $\sigma^2$. By scaling, this is
equivalent to asking whether recovery is possible from the space
$\frac{1}{\sigma}\X:=\{(\frac{1}{\sigma} x_n)\,:\, x\in\X\}$ with
the standard noise level.

In Sections~\ref{sec:unknown_phase} and ~\ref{sec:tree_trail} we
present examples of spaces $\X$ with a ``critical signal-to-noise
ratio'' for detection and recovery. By this we mean that there
exists some critical $\sigma_c$ such that $\X$ admits recovery when
$\sigma<\sigma_c$ and does not even admit detection when
$\sigma>\sigma_c$. In the example of Section~\ref{sec:tree_trail},
recovery is also possible at $\sigma_c$ itself.

\subsection{}

Our concept of recovery is related to the concept of disjointness of
two dynamical systems as introduced by Furstenberg
\cite{furstenberg:disjointness}. Furstenberg showed that in any zero
entropy stochastic process there exists a subset of full measure
which admits recovery. Here is a special case of this result.
Suppose we are given a continuous function $f:\T^d\to\R$ and
$\alpha\in\T^d$, and consider signals of the form $(f(x +
n\alpha))$, $n\ge 0$, indexed by points $x\in\T^d$. Furstenberg's
result implies that recovery is possible from the space of such
signals when $x$ is restricted to a subset of full measure of
$\T^d$. In this case we can strengthen the conclusion and show that
recovery is possible from this space also when $x\in\T^d$ is
unrestricted. In fact, our results imply that recovery is possible
from the much larger space of signals when neither $\alpha$ nor $f$
are specified. More generally, we show that recovery is possible
from the space of all almost periodic functions, which includes the
above examples as a special case.

Let us state our result precisely. We denote by $\AP$ the uniform
closure (i.e., closure in the $\ell^\infty$ norm) of the linear
combinations of functions on $\Z$ of the form $e^{2 \pi i n\theta}$,
$\theta\in\T$. This is a translation invariant algebra of functions.
The following result is proved in Section~\ref{sec:poly_phase}.
\begin{theorem}
\label{thm:reconstruction-polynomial-intro} The space $\AP$ admits
recovery.
\end{theorem}
Section~\ref{sec:poly_phase} also includes a discussion of the
larger space of polynomial phase functions. This space is defined as
the uniform closure of the linear combinations of functions on $\Z$
of the form $e^{2 \pi i p(n)}$, where $p(x)$ is a real polynomial.
It is a translation-invariant algebra of functions, which contains
the almost periodic functions. We give an indication of why recovery
may be possible from this larger space as well.

\subsection{} A natural continuous analogue of our setup is the
following. Consider a complex Brownian motion $B(t)$ on the circle
$\T$, i.e., a continuous random function defined via the Fourier
series
\begin{equation*}
  \xi_0 + \sum_{n\neq 0} \frac{\xi_n}{n} e^{2\pi i n t}
\end{equation*}
for an independent sequence of standard complex Gaussian random
variables $\{\xi_n\}$. Suppose $\X$ is a space of functions on $\T$.
Given a function of the form $f(t)+B(t)$ for some $f\in\X$, is it
possible to recover the function $f$? Is it possible to detect that
$f$ is non-zero? In other words, given a Brownian motion with a
drift taken from a prescribed class, is it possible to identify the
precise drift, or at least to detect that it is there? For certain
spaces $\X$, such questions may be embedded in our standard setup.
Indeed, if $f(t)$ has Fourier series
\begin{equation*}
  x_0 + \sum_{n\neq 0} \frac{x_n}{n} e^{2\pi i n t}
\end{equation*}
then recovering $f$ from $f + B$ is equivalent to recovering
$\{x_n\}$ from $\{x_n + \xi_n\}$.

\subsection{} In Section~\ref{sec:alternative_models} we briefly discuss some variants of our
setup, including alternative noise distributions, a notion of
detection and recovery which is uniform in the transmitted signal
and a notion of partial recovery. The final
Section~\ref{sec:remarks_open_questions} presents several open
problems and remarks.


\section{Definitions and related works}
\label{section:definitions_related_works}

\subsection{} In this section we describe our setup in more detail. We are given
a countable index set $I$, typically $\N$ or $\Z$. A \emph{signal
space} $\X$ is then a subset of $\R^I$, in the real case, or a
subset of $\C^I$, in the complex case. We let $(\xi_n)$, $n\in I$,
be a sequence of independent standard Gaussian random variables.
This means in the real case that $\xi_n\sim N(0,1)$, and means in
the complex case that the real and imaginary parts of $\xi_n$ are
distributed as $N(0,1)$ independently. Thus, our setup is unaffected
if we treat complex signals as real signals by putting the real and
imaginary parts in separate coordinates. For brevity, we continue
the description of the setup in the real case.

We say that the space $\X$ \emph{admits recovery} if there exists a
Borel measurable function $T:\R^I\to\R^I$ with the property that for
each $x\in \X$, we have $T(\{x_n + \xi_n\}) = \{x_n\}$ almost
surely. Here, the probability is taken over the noise sequence. We
say that the space $\X$ \emph{admits detection} if there exists a
Borel measurable function $T:\R^I\to\{0,1\}$ satisfying that for
every $x\in \X$ we have $T(\{x_n + \xi_n\}) = 1$ almost surely, and
$T(\{\xi_n\}) = 0$ almost surely.

\subsection{} As mentioned in the introduction, the idea of
separating signals from noise has been considered previously by many
authors (see \cite{ingster-suslina} and references within).
Typically, however, it has been considered when the signals have
finite length and the detection or recovery probabilities are
strictly less than one. In the context of infinite signals, our work
is especially related to the work of Ingster, see the book
\cite{ingster-suslina}, who considered a similar setup to ours, but
focused on the particular case that the signal space $\X$ is defined
via norm inequalities. Another particular case of our setup has
appeared in the work of Arias-Castro, Cand\'es, Helgason and
Zeitouni \cite{ACCHZ} who considered the problem of detecting a
trail in a graph, see also in Section~\ref{sec:tree_trail}. A
theorem related to our Theorem~\ref{thm:Rademacher_detection} was
shown by Kadota and Shepp \cite{kadota-shepp}. In the context of
signals with finite length, a problem related to our
Theorem~\ref{thm:Fourier_transforms_of_measures} was considered by
Donoho and Jin \cite{donoho-Jin}. Lastly, the effect of introducing
feedback in the transmission channel was considered by Polyanskiy,
Poor and Verd\'u \cite{polyanskiy-poor-verdu}.

%

\section{Non-detection via signal randomization}\label{sec:general_conditions}
\subsection{}Suppose that a signal space $\X$ admits detection. By
definition, this means that there exists a $\{0,1\}$-valued Borel
measurable mapping $T$ satisfying that $T(\{\xi_n\})=0$ almost
surely, and, for every $x\in\X$, $T(\{x_n + \xi_n\})=1$ almost
surely. Now suppose that we are given a probability measure $P$ on
$\X$. It follows from Fubini's theorem that if $x$ is sampled from
$P$, independently of the noise sequence $\xi$, then it is still
true that $T(\{x_n + \xi_n\})=1$ almost surely, where now the
probability is taken over the product space of $x$ and $\xi$.
Equivalently, the distributions of $\xi$ and $x+\xi$ are mutually
singular. It follows that one possibility for showing that $\X$ does
\emph{not} admit detection is to find a probability measure $P$ on
$\X$ such that $\xi$ and $x+\xi$ are not mutually singular. Our next
two theorems exploit this fact to present conditions for detection.
Theorem~\ref{thm:general_detection} provides a relatively simple
sufficient condition for non-detection.
Theorem~\ref{thm:detection_singularity_criterion} shows that in
certain situations, detection may be characterized by this approach,
however, with a more complicated condition.

\begin{theorem}\label{thm:general_detection}
Let $\X$ be a Borel subset of $\R^\N$. Suppose that there is a
probability measure $P$ on $\X$ such that
\begin{equation}\label{eq:non_detection_cond}
  \liminf_{k \to\infty} \; \e \, \exp \bigg\{\sum_{n=1}^k x_n y_n \bigg\} <
  \infty,
\end{equation}
where $\{x_n\}$ and $\{y_n\}$ are sampled independently from $P$.
Then $\X$ does not admit detection.
\end{theorem}
We make several remarks concerning this theorem. First, by choosing
$P$ to be concentrated on a single element, this shows that
$\X\cap\ell^2=\emptyset$ is a necessary condition for detection, as
stated in the introduction.

Second, it follows from the proof of the theorem that the sequence
of expectations in \eqref{eq:non_detection_cond} is non-decreasing
with $k$ and hence the $\liminf$ is in fact a $\lim$.

Lastly, as the proof shows, condition \eqref{eq:non_detection_cond}
is only a sufficient condition for showing that the measures of
$\xi$ and $x+\xi$ are non-singular, when $x$ is sampled from $P$
independently of $\xi$. The precise condition for non-singularity is
that a certain martingale converges to a non-zero limit with
positive probability, whereas condition
\eqref{eq:non_detection_cond} is equivalent to the same martingale
being bounded in $L^2$. More on the gap between these two conditions
can be found in
Section~\ref{sec:gap_between_ell_2_and_uniform_integrability}. The
advantage of condition \eqref{eq:non_detection_cond}, however, is
that it is simple to check in many applications.

Clearly, the choice of index set for the coordinates of $\X$ makes
no difference to the possibility of detection from $\X$. Thus one
may replace $\R^\N$ in the above theorem by $\R^\Z$, or replace the
sum from $1$ to $k$ by sums over arbitrary sets increasing to the
whole index set. Similarly, for complex-valued signals condition
\eqref{eq:non_detection_cond} generalizes to
\begin{equation*}
  \liminf_{k \to\infty} \; \e \, \bigg|\exp \bigg\{\sum_{n=1}^k x_n \overline{y_n} \bigg\}\bigg| <
  \infty
\end{equation*}
by identifying the complex-valued signal space with a real-valued
signal space as explained in
Section~\ref{section:definitions_related_works}.

The next theorem gives a necessary and sufficient condition for
detection, in terms of the possible probability measures on $\X$.
The proof was explained to us by Boris Tsirelson \cite{tsirelson},
following our question to him.
\begin{theorem}\label{thm:detection_singularity_criterion}
  Let $\X$ be a compact subset of $\R^\N$. Then $\X$ admits
  detection if and only if for every probability measure $P$ on
  $\X$, if $x$ is sampled from $P$ independently of $\xi$, then the
  distributions of $\xi$ and $x+\xi$ are mutually singular.
\end{theorem}

\subsection{Proof of
Theorem~\ref{thm:general_detection}}\label{sec:proof_of_general_non_detection_thm}

It is sufficient to show that the distributions of the random
sequences $\{\xi_n\}$ and $\{x_n + \xi_n\}$, where $\{x_n\}$ is
sampled from $P$ independently of $\{\xi_n\}$, are not mutually
singular. Indeed, if it is possible to detect a signal from $\X$ via
some $\{0,1\}$-valued Borel measurable mapping $T$, then, using
Fubini's theorem, the event $\{T=0\}$ has full measure under
$\{\xi_n\}$ and zero measure under $\{x_n + \xi_n\}$. The following
lemma gives a criterion for mutual singularity.
\begin{lemma}[{see, e.g.,\ \cite[p.\ 242]{durrett}}]
\label{lemma:durrett} Let $Q, R$ be two probability measures on a
measu\-rable space $(\Omega, \F)$. Let $\{\F_k\}$ be an increasing
sequence of $\sigma$-fields generating $\F$, and let $Q_k, R_k$ be
the restrictions to $\F_k$ of $Q, R$ respectively. Suppose that
$R_k$ is absolutely continuous with respect to $Q_k$, and let $f_k
:= dR_k / dQ_k$. Then a necessary and sufficient condition for the
measures $Q, R$ to be mutually singular is that $f_k \to 0$
$Q$-almost surely as $k \to \infty$.
\end{lemma}

In the notation of the lemma, we take $(\Omega, \F)$ to be the set
$\R^\N$ equipped with its Borel $\sigma$-field. The elements of
$\Omega$ will be denoted by $z = \{z_n\}$. Let $\F_k \subset \F$
denote the Borel $\sigma$-field generated by the projections onto
the coordinates $z_n$ with $1\le n\le k$. Let $Q$ be the
distribution on $\Omega$ of the noise sequence $\xi = \{\xi_n\}$,
and $R$ be the distribution of the random sequence $\{x_n +
\xi_n\}$, where $\{x_n\}$ is sampled from $P$ independently of
$\{\xi_n\}$. Let also $Q_k$ and $R_k$ be the restrictions to $\F_k$
of the measures $Q$ and $R$ respectively.

When needed to avoid ambiguity, the expectations with respect to the
distributions of $\{\xi_n\}$ and $\{x_n\}$ will be denoted by
$\e_\xi$ and $\e_x$, respectively. We will similarly use $\P_\xi$
and $\P_x$.

Theorem \ref{thm:general_detection} will follow from Lemma
\ref{lemma:durrett} if we show that the Radon-Nikodym derivative
$f_k = dR_k/dQ_k$ does not tend to zero $Q$-almost surely as $k \to
\infty$, or equivalently, that the random variable $f_k(\xi)$ does
not tend to zero almost surely.
We will show in Lemmas~\ref{lemma:radon-nikodym-fk} and
\ref{lemma:moment-fk} below that $f_k$ exists and satisfies
\begin{equation}\label{eq:fk_second_moment}
\e \, f_k(\xi) = 1\quad\text{ and }\quad \e \, f_k(\xi)^2 = \e_x \,
\e_y \, \exp \bigg\{\sum_{n=1}^k x_n y_n \bigg\},
\end{equation}
where $\{x_n\}$ and $\{y_n\}$ are sampled independently from $P$. By
the Paley-Zygmund inequality (see, e.g., \cite[p.
8]{kahane-random}),
\[
\P(f_k(\xi)\ge 1/2)\ge \frac{1}{4\e \, f_k(\xi)^2}.
\]
Thus, under the condition \eqref{eq:non_detection_cond}, this
probability is bounded below uniformly on some subsequence
$k_j\to\infty$ and hence $f_k(\xi)$ does not tend to 0 almost
surely, as required.

It remains only to prove the following two lemmas.

\begin{lemma}
\label{lemma:radon-nikodym-fk} $R_k$ is absolutely continuous with
respect to $Q_k$, and the Radon-Nikodym derivative $f_k = dR_k/dQ_k$
is given by
\begin{equation}
\label{eq:radon-nikodym-fk} f_k(z) = \e_x \, \exp \Big\{
\sum_{n=1}^k \Big(-\frac{x_n^2}{2} + x_n \, z_n\Big) \Big\},
\end{equation}
where $\{x_n\}$ is sampled from $P$.
\end{lemma}

\begin{proof}
The measures $Q_k, R_k$ are both absolutely continuous with respect
to the product Lebesgue measure in the coordinates $1\le n\le k$,
and they are given by
\[
dQ_k(z) = \prod_{n=1}^k \frac1{\sqrt{2\pi}} \, \exp \Big\{
-\frac1{2} z_n^2 \Big\} \, dz_n
\]
and
\[
dR_k(z) = \e_x \left[\prod_{n=1}^k \frac1{\sqrt{2\pi}} \, \exp
\Big\{ -\frac1{2} \Big( z_n - x_n \Big)^2 \, \Big\}\, dz_n\right].
\]
Hence the Radon-Nikodym derivative $f_k$ exists and satisfies
\eqref{eq:radon-nikodym-fk}.
\end{proof}

\begin{lemma}
\label{lemma:moment-fk} The random variable $f_k(\xi)$ satisfies
\eqref{eq:fk_second_moment}.
\end{lemma}

\begin{proof}
The fact that the expectation of $f_k(\xi)$ is equal to $1$ follows
from the definition of $f_k$ as the Radon-Nikodym derivative
$dR_k/dQ_k$. To calculate the second moment, we use
Lemma~\ref{lemma:radon-nikodym-fk} to obtain
\begin{equation*}
  f_k(z)^2 = \e_x \, \e_y \, \exp \Big\{ \sum_{n=1}^k
\Big(-\frac{x_n^2 + y_n^2}{2} + \big(x_n+y_n\big) \, z_n\Big)
\Big\},
\end{equation*}
where $\{x_n\}$ and $\{y_n\}$ are sampled independently from $P$.
Thus
\begin{equation}\label{eq:exp_f_k_squared}
\e_\xi \, f_k(\xi)^2 = \e_x \, \e_y \, \Big[\exp \Big\{
-\sum_{n=1}^k \frac{x_n^2 + y_n^2}{2}\Big\} \, \prod_{n=1}^k \e_\xi
\, \exp\Big\{\big(x_n+y_n\big) \, \xi_n \Big\}\Big],
\end{equation}
where we have used the independence of the $\xi_n$. Formula
\eqref{eq:fk_second_moment} now follows upon observing that
\begin{equation*}
\e_\xi \, \exp\Big\{\big(x_n+y_n\big) \, \xi_n \Big\} =
\exp\Big\{\frac{\big(x_n+y_n\big)^2}{2}\Big\}.\qedhere
\end{equation*}
\end{proof}
Although not necessary for the above proof, it is instructive to
make the following observations. Note that, by definition, the
sequence $(f_k(\xi))$ is a positive martingale and hence converges
almost surely. Since $\e f_k(\xi) = 1$, a sufficient condition for
the limit to not be identically zero is that the martingale be
uniformly integrable (in fact, it also implies that $R$ is
absolutely continuous with respect to $Q$). Our proof establishes
this by showing that under condition~\eqref{eq:non_detection_cond},
$(f_k(\xi))$ is even bounded in $L^2$. Finally, recalling that the
square of a martingale is a submartingale we see that $\e
f_k(\xi)^2$ is non-decreasing in $k$. Thus the $\liminf$ in
\eqref{eq:non_detection_cond} may be replaced by a $\lim$.

\subsection{Proof
of Theorem~\ref{thm:detection_singularity_criterion}} The only if
part of the theorem was explained in the beginning of the section.
Therefore we focus on proving the if part. The proof uses Sion's
minimax theorem, a special case of which we now cite.
\begin{theorem}(Special case of Sion's minimax theorem \cite{sion})\label{thm:Sion}
Let $K$ be a compact convex subset of a linear topological space and
$V$ be a convex subset of a linear topological space. If $f:K\times
V\to\R$ is a continuous bilinear mapping then
\begin{equation*}
  \min_{x\in K} \sup_{y\in V} f(x,y) = \sup_{y\in V} \min_{x\in K}
  f(x,y).
\end{equation*}
\end{theorem}
Since $\X$ is compact, the space of probability measures on $\X$ is
convex and compact, under the topology of weak convergence of
measures. For each probability measure $P$ on $\X$, let $\tilde{P}$
be the probability measure on $\R^\N$ obtained as the distribution
of $x+\xi$, where $x$ is sampled from $P$, independently of $\xi$.
Let $K$ be the space of all probability measures $\tilde{P}$ as $P$
ranges over all probability measures on $\X$. It follows that $K$ is
convex and compact  (under the same topology), since the mapping $P
\mapsto \tilde{P}$ is linear and continuous. Now, fix $\eps>0$ and
let $V_\eps$ be the space of continuous functions $g:\R^\N\to[0,1]$
satisfying $\e g(\xi)\le \eps$. Observe that $V_\eps$ is a convex
subset of the space of bounded continuous functions on $\R^\N$ (with
the sup-norm). Let $f:K\times V_\eps\to\R$ be the expectation
operator, defined by $f(\tilde{P},g) := \int g \, d\tilde{P}$. Since
$f$ is a bilinear continuous mapping, we may apply
Theorem~\ref{thm:Sion} and conclude that
\begin{equation}\label{eq:Sion_application}
   \min_{\tilde{P}\in K} \sup_{g\in V_\eps} \int g \, d\tilde{P}
   = \sup_{g\in V_\eps} \min_{\tilde{P}\in K} \int g \,  d\tilde{P}.
\end{equation}
Now, if we assume that the distribution of $\xi$ is singular to
every measure $\tilde{P}$ in $K$, then the left-hand side of
\eqref{eq:Sion_application} equals 1. Thus, the right-hand side also
equals 1, and we conclude that there exists a function $g\in V_\eps$
such that
\begin{equation*}
  \min_{\tilde{P}\in K} \int g \, d\tilde{P} \ge 1-\eps.
\end{equation*}
Since $\eps>0$ is arbitrary, this implies that there exists a
sequence of continuous functions $g_n:\R^\N\to[0,1]$ satisfying
\begin{equation}\label{eq:g_n_def}
  \e g_n(\xi) \le 2^{-n}\quad\text{ and }\quad
  \min_{\tilde{P}\in K} \int g_n \, d\tilde{P} \ge 1 - 2^{-n}.
\end{equation}
Define a sequence of measurable functions $h_n:\R^\N\to[0,1]$ by
$h_n(x):=\inf_{k\geq n} g_k(x)$. Then $h_n(x)\le g_n(x)$ and hence
it follows that $\e h_n(\xi)\le 2^{-n}$. On the other hand, by the
second inequality in \eqref{eq:g_n_def}, it follows that $\int h_n\,
d\tilde{P} \ge 1-2^{-n+1}$ for all $\tilde{P}\in K$. Finally, since
$h_n(x)$ is an increasing, bounded sequence of functions, it
converges to a measurable limit $h:\R^\N\to[0,1]$ satisfying
\begin{equation}\label{eq:h_properties}
\e h(\xi)=0\quad\text{ and }\quad\min_{\tilde{P}\in K} \int h \,
d\tilde{P} = 1.
\end{equation}
We may thus use $h$ to show that $\X$ admits detection. Indeed,
given a noisy signal $z$ we may distinguish the two cases $z = \xi$
and $z = x + \xi$ for some $x\in\X$ according to whether $h(z)=0$ or
$h(z)=1$. To see this, observe that if $z=\xi$ then by
\eqref{eq:h_properties}, $\P(h(z)=0)=1$, where the probability is
over $\xi$. In addition, if $z = x+\xi$ for some $x\in\X$ then,
since the distribution of $z$ is in $K$, \eqref{eq:h_properties}
implies that $\P(h(z)=1) = 1$. \qed


\section{Fourier transforms and Hausdorff dimension}
\label{sec:hausdorff}

\subsection{}
In this section we consider signals which are Fourier transforms of
finite, complex Borel measures on the circle group $\T = \R / \Z$.
The Fourier transform $\{\ft{\mu}(n)\}$ of a measure $\mu$ on $\T$
is given by
\[
\ft{\mu}(n) = \int_{\T} e^{-2 \pi i n t} \, d\mu(t), \quad n \in \Z.
\]

We say that $\mu$ is carried by a Borel set $E$ if $|\mu|(\T
\setminus E) = 0$. Our motivating idea is that measures carried by
``small'' subsets of the circle will have some ``redundancy'' in
their Fourier transforms. Thus we may hope that by restricting to
measures with small support, in various senses, we will obtain
spaces of signals admitting recovery.

In the first part of the section we consider the space of Fourier
transforms of all measures carried by a given subset $E$ of the
circle. We show that this space admits recovery if $E$ has Lebesgue
measure zero. It is easy to see that this condition is sharp: if $E$
has positive Lebesgue measure, then distinct measures on $E$ may
differ by an $L^2$ function, violating the necessary condition
\eqref{eq:recover_nec_cond} by Parseval's equality. Similarly, if
$E$ has positive Lebesgue measure, the space will not admit
detection since condition \eqref{eq:detection_nec_cond} will be
violated.

In the second part of the section we consider all measures which are
carried by sets of Hausdorff dimension no larger than a given number
$\alpha$. We find that this space admits recovery if $\alpha<1/2$
and does not even admit detection if $\alpha>1/2$. The case
$\alpha=1/2$ is left open. In what follows $\mes(E)$ denotes the
Lebesgue measure of $E$ and $\dim(\mu)$ is the dimension of a
measure $\mu$ defined in \eqref{eq:dimension_of_mu}.

\begin{theorem}
\label{thm:known_support} Let $E\subset\T$ be a Borel set and let
$\F_E$ consist of the Fourier transforms of all finite, complex
measures carried by $E$. Then
\begin{enumerate-math}
  \item\label{part:thm_known_support_part_I} If $\mes(E)=0$ then $\F_E$ admits recovery.
  \item\label{part:thm_known_support_part_II} If $\mes(E)>0$ then $\F_E\setminus\{0\}$ does not admit detection.
\end{enumerate-math}
\end{theorem}

\begin{theorem}
\label{thm:reconstruction-small-dimension} Let $\F_\alpha$ consist
of the Fourier transforms of all finite, complex measures $\mu$,
such that $\dim(\mu)\le\alpha$. Then
\begin{enumerate-math}
  \item\label{it:recovery_measures} If $\alpha<1/2$ then $\F_\alpha$ admits recovery.
\item\label{it:non-detection_measures} If $\alpha>1/2$ then $\F_\alpha\setminus\{0\}$ does not admit
detection.
\end{enumerate-math}
\end{theorem}
The rest of the section is devoted to the proofs of these two
theorems.

\subsection{Integration of noisy signals}
To each Borel set $E\subset \T$ we associate the random variable
\begin{equation*}
  \xi(E) := \sum_{n\in\Z} \ft{\1}_E(-n)\xi_n.
\end{equation*}
The sum converges in $L^2$ and almost surely. Recalling that in our
convention $\e|\xi_0|^2 = 2$ when the noise is complex, we see that
$\xi(E)$ is a centered complex Gaussian random variable with
$\e|\xi(E)|^2 = 2\mes(E)$.

Intuitively, $\xi(E)$ is the ``integral'' over $E$ of the formal
Fourier series $\sum \xi_n e^{2\pi int}$, which may be seen as white
noise on the circle.

\begin{lemma}\label{lem:T_E_prop}
  Let $E\subset\T$ be either an open or closed set. There exists a Borel measurable mapping $T_E:\C^\Z\to\C$ such
that, almost surely,
\begin{equation*}
  T_E(\ft{\mu} + \xi) = \int_E d\mu + \xi(E)
\end{equation*}
for every complex Borel measure $\mu$ on $\T$.
\end{lemma}
\begin{proof}
We choose a sequence $\psi_j$ of smooth functions on $\T$ such that
\begin{equation*}
  0\le \psi_j\le 1,\quad\|\psi_j-\1_E\|_{L^2(\T)}\le 1/j,\quad \psi_j(t)\to
  \1_E(t)\quad \forall t\in\T.
\end{equation*}
This is possible since $E$ is either an open or closed set. The
mapping $T_E$ is defined by
\begin{equation*}
  T_E(y):=\lim_{j\to\infty} \sum_{n\in\Z} \ft{\psi}_j(-n)y_n,
\end{equation*}
where, if the limit does not exist, we define $T_E(y)=0$. It is
simple to check that this is a Borel measurable mapping.

Let us first show that $T_E(\ft{\mu}) = \int_E d\mu$ for any complex
Borel measure $\mu$ on $\T$. Indeed, since $\psi_j$ is smooth, we
have
\begin{equation*}
  \sum_{n\in\Z} \ft{\psi}_j(-n)\ft{\mu}(n) = \int_\T \psi_j d\mu,
\end{equation*}
where the series on the left-hand side converges absolutely. Now,
the bounded convergence theorem implies that $T_E(\ft{\mu}) = \int_E
d\mu$.

Next, we show that $T_E(\xi) = \xi(E)$ almost surely. Observe that
\begin{equation*}
  \eta_j := \sum_{n\in\Z} \ft{\psi}_j(-n)\xi_n
\end{equation*}
converges almost surely and defines a complex Gaussian random
variable. Moreover,
\begin{equation*}
  \e|\eta_j - \xi(E)|^2 = 2\sum_{n\in\Z} |\ft{\psi}_j(n) -
  \ft{\1}_E(n)|^2 = 2\|\psi_j-\1_E\|_{L^2(\T)}^2 \le \frac{2}{j^2}.
\end{equation*}
An application of the Borel-Cantelli lemma shows that $\eta_j\to
\xi(E)$ almost surely.

Finally, the lemma follows from the linearity of $T_E$.
\end{proof}

\subsection{Recovery when the support is known. Proof of Theorem~\ref{thm:known_support}} Let us fix a Borel set $E \subset \T$.
Part~\ref{part:thm_known_support_part_II} of the theorem has already
been explained so it remains to prove
Part~\ref{part:thm_known_support_part_I}.

We thus suppose that $E$ has Lebesgue measure zero, and take a
signal from $\F_E$, which is of the form $\{\ft{\mu}(n)\}$ where
$\mu$ is a Borel measure carried by $E$. The noisy version of the
signal will be denoted by
\begin{equation}\label{eq:noisy_measure}
y(n) = \ft{\mu}(n) + \xi(n).
\end{equation}
We need a recovery algorithm for $\ft{\mu}(n)$. Certainly it
suffices to present such an algorithm for $\ft{\mu}(0) = \int d\mu$,
since the noise distribution and the support of the measure are
unchanged when shifting $y$. We choose a sequence of open sets $U_j$
such that $E\subset U_j$ and $\mes(U_j) < j^{-2}$. Our recovery
procedure consists of calculating
\[
T(y):=\lim_{j\to \infty} T_{U_j}(y),
\]
where $T_{U_j}$ is given by Lemma~\ref{lem:T_E_prop} and, if the
limit does not exist, we define $T(y)=0$. Certainly $T$ is a Borel
measurable mapping since $T_{U_j}$ is Borel measurable for every
$j$. We claim that, almost surely,
\begin{equation*}
  T(\ft{\mu} + \xi) = \int d\mu
\end{equation*}
for every measure $\mu$ carried by $E$. Indeed, by
Lemma~\ref{lem:T_E_prop}, almost surely,
\begin{equation*}
  T_{U_j}(\ft{\mu} + \xi) = \int_{U_j} d\mu + \xi(U_j) = \int_{E} d\mu + \xi(U_j)
\end{equation*}
for every measure $\mu$ carried by $E$. Thus the claim follows by
noting that, almost surely,
\begin{equation*}
  \xi(U_j)\to 0
\end{equation*}
which follows from our assumption that $\mes(U_j)<j^{-2}$ and the
Borel-Cantelli lemma.\qed

\subsection{Recovery with unknown support. Proof of part
\ref{it:recovery_measures} of
Theorem~\ref{thm:reconstruction-small-dimension}} As $\F_\alpha$ is
increasing in $\alpha$ it suffices to prove
part~\ref{it:recovery_measures} with $\le$ replaced by $<$ in the
definition of $\F_\alpha$. Fix $\alpha<1/2$. An interval
$I\subset\T$ of the form $[\frac{k}{2^N}, \frac{k+1}{2^N}]$ will be
called a dyadic interval of rank $N$, and we denote $\rank(I)=N$. We
denote by $\D_N$ the family of all sets which are unions of dyadic
intervals of rank $N$. We also define
\begin{equation*}
\D_N^\alpha:=\left\{E\in\D_N\,:\, \mes(E)\le
2^{-(1-\alpha)N}\right\}
\end{equation*}
and a random variable,
\begin{equation*}
Y_N:=\max\left\{|\xi(E)|\,:\, E\in \D_N^\alpha\right\}.
\end{equation*}
\begin{lemma}
For all $N\ge N(\alpha)$ we have
\begin{equation*}
  \P(Y_N\ge \exp(-cN))\le \exp\left(-c2^{\frac{1}{2}N}\right),
\end{equation*}
where $c>0$ is a constant depending only on $\alpha$.
\end{lemma}
\begin{proof}
  We recall that if $Z$ is a complex Gaussian random variable, with
  independent real and imaginary parts of variance $\sigma^2$, then
  \begin{equation}\label{eq:complex_Gaussian_bound}
    \P(|Z|>t) = \exp\left(-t^2/(2\sigma^2)\right).
  \end{equation}
  Observe that
  \begin{equation}\label{eq:K_equality}
    |\D_N^\alpha|=\sum_{j=0}^{\lfloor 2^{\alpha N}\rfloor} \binom{2^N}{j}\le
    (\lfloor 2^{\alpha N} \rfloor + 1) \binom{2^N}{\lfloor 2^{\alpha
    N}\rfloor} \le (\lfloor 2^{\alpha N} \rfloor + 1) 2^{N2^{\alpha N}} \le \exp\left(N2^{\alpha N}\right),
  \end{equation}
  where we have used the inequality $\binom{n}{k}\le n^k$ valid for $n,k\ge 1$. Since for
  each $E\in \D_N^\alpha$ we have
  $\e|\xi(E)|^2 \le 2^{1-(1-\alpha)N}$, a union bound using \eqref{eq:complex_Gaussian_bound} gives
  \begin{equation*}
    \P(Y_N \ge t) \le |\D_N^\alpha|e^{-t^2 2^{(1-\alpha)N-1}} \le \exp\left(N2^{\alpha N}-t^2 2^{(1-\alpha)N-1}\right).
  \end{equation*}
  Substituting $t=2^{-\frac{1}{2}(\frac{1}{2} - \alpha)N}$ finishes
  the proof.
\end{proof}
Define a random variable
\begin{equation}\label{eq:Z_N_def}
  Z_N := \sum_{n\ge N} Y_n
\end{equation}
and observe that by the Borel-Cantelli lemma and the previous lemma,
\begin{equation}\label{eq:Z_N_small_prop}
  Z_N \le \exp(-cN)\quad\text{for all but finitely many $N$, with probability
  one}.
\end{equation}

As before, the noisy signal has the form \eqref{eq:noisy_measure},
where now $\mu$ is a complex measure carried by a Borel set $E$ of
Hausdorff dimension $<\alpha$. We need a recovery algorithm for
$\ft{\mu}(n)$. Certainly it suffices to present such an algorithm
for $\ft{\mu}(0) = \int d\mu$.

A complex measure $\mu$ has a unique decomposition
\begin{equation*}
\mu = \mu_1 - \mu_2 + i\mu_3 - i\mu_4,
\end{equation*}
where each $\mu_j$ is a positive measure, where $\mu_1$ and $\mu_2$
are mutually singular, and the same for $\mu_3$ and $\mu_4$. We
describe an algorithm for recovering $\int d\mu_1$. One may recover
$\int d\mu_j$ in a similar way.

The recovery algorithm is as follows.

Given $N_1<N_2$, let $\Omega_{N_1,N_2}$ denote the (finite) family
of all sets $E$ admitting a representation of the form $E=\cup I_j$
where each $I_j$ is a dyadic interval with $N_1\le \rank(I_j)\le
N_2$, and such that $\sum |I_j|^\alpha \le 1$. We define the mapping
\begin{equation*}
  T(y) := \limsup_{N_1\to\infty} \limsup_{N_2\to\infty}
  \max_{E\in \Omega_{N_1, N_2}} \Re\left(T_E(y)\right),
\end{equation*}
where $T_E$ is the mapping given by Lemma~\ref{lem:T_E_prop}. Here,
$\Re(z)$ denotes the real part of a complex number $z$. Certainly
$T$ is a Borel measurable mapping since $T_E$ is Borel measurable
for every $E$.

We claim that, almost surely,
\begin{equation}\label{eq:measure_recovery_algorithm}
  T(\ft{\mu} + \xi) = \int d\mu_1
\end{equation}
for any complex Borel measure $\mu$ carried by a Borel set $S$ of
Hausdorff dimension $<\alpha$.

We now prove \eqref{eq:measure_recovery_algorithm}. We start with a
deterministic lemma.
\begin{lemma}\label{lem:mu_recovery_lemma}
  Let $\mu$ be a complex Borel measure, carried by a Borel set $S$ of
  Hausdorff dimension $<\alpha$. Then
  \begin{equation*}
    \limsup_{N_1\to\infty} \limsup_{N_2\to\infty}
  \max_{E\in \Omega_{N_1, N_2}} \Re\left(\int_E d\mu\right) = \int
  d\mu_1.
  \end{equation*}
\end{lemma}
\begin{proof}
  We first note that for any Borel set $E$,
  \begin{equation*}
    \Re\left(\int_E d\mu\right) \le \int d\mu_1.
  \end{equation*}
  Since $\mu_1$ and $\mu_2$ are mutually singular, there exist two disjoint Borel sets $S_1, S_2\subset S$ such that
  $\mu_1$ is carried by $S_1$ and $\mu_2$ is carried by $S_2$. Given $\eps>0$ we choose two compact sets $K_1\subset S_1$ and $K_2\subset S_2$
  such that
  \begin{equation*}
    \int_{K_j} d\mu_j \ge \int d\mu_j - \eps\quad j=1,2.
  \end{equation*}
  Given $N_1$, since $K_1$ is compact and has Hausdorff dimension $<\alpha$ we may cover it by
  a finite number of dyadic intervals $(I_j)$ such that $\sum
  |I_j|^\alpha \le 1$, $\rank(I_j)\ge N_1$ and $I_j\cap K_2=\emptyset$ for all $j$. Then
  the set $E = \cup I_j$ belongs to $\Omega_{N_1, N_2}$ for some
  $N_2$ and
  \begin{equation*}
    \Re\left(\int_E d\mu\right) = \int_E d\mu_1 - \int_E d\mu_2 \ge \int d\mu_1 -
    2\eps.
  \end{equation*}
  Since $N_1$ and $\eps$ are arbitrary, the lemma follows.
\end{proof}
\begin{lemma}\label{lem:max_noise_on_Omega}
  Given $N_1<N_2$, almost surely,
  \begin{equation*}
    \max_{E\in\Omega_{N_1, N_2}} |\xi(E)| \le Z_{N_1}.
  \end{equation*}
\end{lemma}
\begin{proof}
  Let $E\in\Omega_{N_1, N_2}$. We may choose a representation $E =
  \cup I_j^N$ where each $I_j^N$ is a dyadic
  interval of rank $N$, $N_1\le N\le N_2$, the $I_j^N$ have disjoint interiors, and $\sum
  |I_j^N|^\alpha \le 1$. Now let $E_N:=\cup_j I_j^N$ and observe
  that $E_N\in\D_N^\alpha$. Since the $E_N$ are disjoint up to measure zero,
  we have, almost surely,
  \begin{equation*}
    \xi(E) = \sum_{N=N_1}^{N_2} \xi(E_N).
  \end{equation*}
  Finally, the lemma follows from the definition~\eqref{eq:Z_N_def}
  of $Z_N$.
\end{proof}
It follows from Lemmas~\ref{lem:T_E_prop} and
\ref{lem:max_noise_on_Omega} that for any $N_1<N_2$, almost surely,
\begin{equation*}
  \left|\max_{E\in \Omega_{N_1, N_2}} \Re\left(T_E(\ft{\mu} + \xi)\right)
  - \max_{E\in \Omega_{N_1, N_2}} \Re\left(\int_E d\mu\right)\right|
  \le Z_{N_1}
\end{equation*}
for every complex Borel measure $\mu$. The equality
\eqref{eq:measure_recovery_algorithm} now follows from
Lemma~\ref{lem:mu_recovery_lemma} and \eqref{eq:Z_N_small_prop},
proving the correctness of the recovery algorithm and completing the
proof of part \ref{it:recovery_measures} of
Theorem~\ref{thm:reconstruction-small-dimension}.

\subsection{Non-detection with unknown support. Proof of part \ref{it:non-detection_measures} of Theorem~\ref{thm:reconstruction-small-dimension}}
In this section we prove that $\F_\alpha$ does not admit detection
when $\alpha>1/2$, establishing part \ref{it:non-detection_measures}
of Theorem~\ref{thm:reconstruction-small-dimension}. The result will
follow from our general non-detection condition
(Theorem~\ref{thm:general_detection}). However, it will be
convenient to apply the condition not to the space $\F_\alpha$
directly but to the space $\Wa_\alpha$ (defined below) of Walsh
transforms of the corresponding measures. As we will prove,
detection in the space $\Wa_\alpha$ is equivalent to detection in
$\F_\alpha$.

Let $q\ge 2$ be a fixed integer. Each non-negative integer $n$
admits a unique representation in base $q$ in the form
\begin{equation}\label{eq:n_base_q}
  n = \sum_{j\ge 1} n_j q^{j-1}, \quad n_j\in\{0,1,\ldots,
  q-1\},\quad \sum n_j<\infty.
\end{equation}
Similarly, any $t\in[0,1)$ admits an expansion
\begin{equation}\label{eq:t_base_q}
  t = \sum_{j=1}^\infty t_j q^{-j},\quad t_j\in\{0,1,\ldots,
  q-1\}.
\end{equation}
We let $\{w_n(t)\}$, $n\ge 0$, denote the Walsh functions in base
$q$, defined by
\begin{equation*}
  w_n(t) = \exp\left\{\frac{2\pi i}{q}\sum_{j\ge 1} n_j
  t_j\right\}.
\end{equation*}
As a convention we define $w_n(t)$ to be continuous from the right,
avoiding ambiguity when $t$ admits two different expansions
\eqref{eq:t_base_q}. It is well-known that the Walsh system
$\{w_n(t)\}$ forms a complete orthonormal system in the space
$L^2(\T)$. If $\mu$ is a finite measure on $\T$ then it is uniquely
determined by its Walsh coefficients
\begin{equation*}
  \dotprod{\mu}{w_n} := \int_\T \overline{w_n(t)} \, d\mu(t),\quad n\ge 0.
\end{equation*}
To avoid confusion, in this section we will denote the Fourier
coefficients of $\mu$ by
\begin{equation*}
  \dotprod{\mu}{e_n} := \int_\T \overline{e_n(t)} \, d\mu(t),\quad n\in\Z,
\end{equation*}
where $e_n(t) = \exp(2\pi int)$.

We will use the following property of the Walsh system in base $q$.
If $f(t)$ is an integrable function on $\T$ then the partial sum
\begin{equation}\label{eq:partial_sum}
  \sum_{n=0}^{q^s-1} \dotprod{f}{w_n}\, w_n(t)
\end{equation}
is constant on each $q$-adic interval $\big[j/q^s, (j+1)/q^s \big)$
and coincides there with the mean value of $f$ on this interval. In
particular, if $f$ is continuous on $\T$ then the sums
\eqref{eq:partial_sum} converge to $f$ uniformly as $s\to\infty$.

The following proposition shows that detection of Fourier
coefficients is equivalent to detection of Walsh coefficients.
Intuitively, this follows from the fact that the change-of-basis
operator between two orthonormal bases of $\ell^2$ is unitary, and
from the fact that the distribution of a sequence of independent
standard Gaussian random variables is preserved under unitary
transformations.

\begin{proposition}\label{prop:Fourier_Walsh}
  Let $\A$ be a family of measures. If the space $\F_\A$ of Fourier
  transforms of all measures $\mu\in \A$ admits detection, then the
  same is true for the space $\Wa_\A$ of Walsh transforms of the
  measures $\mu\in\A$.
\end{proposition}
\begin{proof}
  Let $T$ be a detection mapping for $\F_\A$. According to
  \eqref{eq:partial_sum}, the exponential function $e_n(t)$,
  $n\in\Z$, admits an expansion
  \begin{equation}\label{eq:e_n_expansion}
    e_n(t) = \lim_{s\to\infty} \sum_{k=0}^{q^s-1} c_{nk} w_k(t)
  \end{equation}
  uniformly convergent on $\T$. We define a measurable mapping $S$
  from $\C^\infty$ to $\C^\infty$ by $S(y) := y'$, where
  \begin{equation*}
    y_n' := \lim_{s\to\infty} \sum_{k=0}^{q^s-1} \overline{c_{nk}} y_k
  \end{equation*}
  if the limit exists, otherwise $y'$ may be arbitrarily defined
  (e.g., by taking limsup, etc.). Let $y_n = \dotprod{\mu}{w_n} + \xi_n$ be
  a noisy version of a signal in $\Wa_\A$. Then by
  \eqref{eq:e_n_expansion},
  \begin{equation}\label{eq:e_n_coeff_expansion}
    \lim_{s\to\infty} \sum_{k=0}^{q^s-1} \overline{c_{nk}} \dotprod{\mu}{w_k}
    = \dotprod{\mu}{e_n}.
  \end{equation}
  On the other hand, the limit
  \begin{equation*}
    \lim_{s\to\infty} \sum_{k=0}^{q^s-1} \overline{c_{nk}} \xi_k =:\xi_n'
  \end{equation*}
  exists almost surely and defines a sequence of random variables $\{\xi_n'\}$ which are also
  Gaussian and independent due to the unitarity of the operator
  $\{c_{nk}\}$.
  It follows that $y' = S(y)$ is a noisy version of the Fourier
  transform of $\mu$. Hence it is clear that $T\circ S$ provides a
  detection mapping for the space $\Wa_\A$.
\end{proof}

\begin{remark*}
  One can show that also the converse to
  Proposition~\ref{prop:Fourier_Walsh} holds, namely, detection in
  the space $\Wa_{\A}$ implies detection in $\F_{\A}$. This can be
  proved in a similar way, exchanging the roles of the Walsh system
  $\{w_n\}$ and the trigonometric system $\{e_n\}$. One may need to
  replace the $q$-adic partial sums in \eqref{eq:e_n_expansion} by
  Fej\'er sums in the corresponding representation of $w_n(t)$ by a
  trigonometric series, in order to justify the corresponding
  \eqref{eq:e_n_coeff_expansion}. We leave the details to the
  reader.
  We also remark that similarly one can show that \emph{recovery} in
  the spaces $\F_\A$ and $\Wa_\A$ is equivalent, in the same way.
\end{remark*}

Let $\Wa_\alpha$ consist of the Walsh transforms of all positive
measures $\mu$, such that $\mu$ is carried by a Borel set of
Hausdorff dimension $\le \alpha$. In view of
Proposition~\ref{prop:Fourier_Walsh}, it will be enough to prove
that detection is not possible in $\Wa_\alpha$ if $\alpha>1/2$. This
will be done below, based on Theorem~\ref{thm:general_detection}.

Given $\alpha>1/2$, we will construct a probability distribution on
the space of measures $\mu$ on $\T$ which are carried by a set of
Hausdorff dimension $\le \alpha$. We choose and fix two positive
integers $q=q(\alpha)$ and $p=p(\alpha)$, such that
\begin{equation}\label{eq:p_and_q}
  q^{1/2} < p < q^{\alpha}.
\end{equation}
We partition the interval $[0,1)$ into $q$ consecutive intervals of
equal length, and choose $p$ of them randomly and uniformly among
the $\binom{q}{p}$ possible choices. We thus obtain a system
$\{I_j^1\}$ of $p$ intervals, each of length $1/q$. To each one of
these intervals we apply a similar procedure: we partition each
$I_j^1$ into $q$ equal length intervals, and choose $p$ of them
randomly and independently of the other choices. We obtain a system
$\{I_j^2\}$ of $p^2$ intervals of length $1/q^2$ each. Continuing
this way, on the $k$'th stage we obtain a (random) system
$\{I_j^k\}$ of $p^k$ intervals of length $1/q^k$. It follows that
the set
\begin{equation*}
  E = \bigcap_{k=1}^\infty \bigcup_j I_j^k
\end{equation*}
is a (random) subset of $\T$ whose Hausdorff dimension is not
greater than $\log p / \log q$ (in fact, the dimension equals $\log
p /\log q$ almost surely, but we do not use this). According to
\eqref{eq:p_and_q}, the Hausdorff dimension of $E$ is $< \alpha$.

The set $E$ carries a natural probability measure $\mu$, which by
definition assigns mass $1/p^k$ to each one of the intervals
$\{I_j^k\}$ of the $k$'th step. By considering the Walsh transform
(in base $q$) of this measure, $x_n = \dotprod{\mu}{w_n}$, as a
random element of $\Wa_\alpha$, we obtain a probability distribution
$\P$ on $\Wa_\alpha$.
\begin{proposition}\label{prop:measure_non-detection}
  We have
  \begin{equation}\label{eq:measure_non-detection}
    \liminf_{k\to\infty} \; \e \,\left|\exp\left\{\sum_{n=0}^{k-1}
    x_n\overline{y_n}\right\}\right| <\infty\, ,
  \end{equation}
  where $x$ and $y$ are sampled independently from the measure $\P$ above.
\end{proposition}

Part~\ref{it:non-detection_measures} of
Theorem~\ref{thm:reconstruction-small-dimension} follows from
Propositions~\ref{prop:Fourier_Walsh},
\ref{prop:measure_non-detection} and
Theorem~\ref{thm:general_detection}.
\begin{proof}[Proof of Proposition~\ref{prop:measure_non-detection}]
Let $\{x_n=\dotprod{\mu}{w_n}\}$ and $\{y_n=\dotprod{\nu}{w_n}\}$ be
the Walsh transforms of two measures $\mu$ and $\nu$, respectively,
constructed using two independent iterations of the random process
above. For each $k$, the partial sums
\begin{equation*}
  \varphi_k(t) = \sum_{n=0}^{q^k-1} \dotprod{\mu}{w_n}\, w_n(t)
\end{equation*}
and
\begin{equation*}
  \psi_k(t) = \sum_{n=0}^{q^k-1} \dotprod{\nu}{w_n}\, w_n(t)
\end{equation*}
are constant on each of the intervals chosen on the $k$'th step of
the construction, and are equal to $(\frac{q}{p})^k$ on these
intervals. By Parseval's equality,
\begin{equation*}
  \sum_{n=0}^{q^k-1} x_n \, \overline{y_n} = \sum_{n=0}^{q^k-1}
  \dotprod{\varphi_k}{w_n} \, \overline{\dotprod{\psi_k}{w_n}} = \int_\T
  \varphi_k(t) \, \overline{\psi_k(t)} \, dt =
  \left(\frac{q}{p^2}\right)^k\cdot Z_k\, ,
\end{equation*}
where $Z_k$ is the number of intervals of the $k$'th step common to
both $\mu$ and $\nu$. By the definition of the random process above,
the random variable $Z_1$ has the hypergeometric distribution,
obtained when sampling $p$ times without replacement from an urn
with $q$ balls, of which $p$ balls are red and $q-p$ balls are
black, and counting the number of red balls sampled.

By the same definition, the random variable $Z_k$ has the
distribution of the population size at the $k$'th generation of a
Galton-Watson branching process, whose offspring distribution has
the law of $Z_1$.

These observations allow us to conclude
\eqref{eq:measure_non-detection} easily using the theory of
branching processes. Indeed, the expected number of offsprings is
\begin{equation*}
  \e(Z_1) = p^2/q > 1,
\end{equation*}
according to \eqref{eq:p_and_q}. Thus, the expectation in
\eqref{eq:measure_non-detection} can be expressed as $\e \exp(W_k)$,
where $W_k$ is the normalized population size, $W_k = Z_k / \e
(Z_k)$. It is well-known that $W_k$ is a non-negative martingale
convergent almost surely to a limit $W$. Moreover, since the
offspring distribution $Z_1$ has finite support, $W_k$ converges
also in $L^1$ to $W$ (see, e.g., \cite[Thm. 2.1]{harris}). By the
conditional Jensen's inequality we have
\begin{equation*}
  \e[\exp (W)\ |\ W_k] \ge \exp \,\e[W\ |\ W_k] = \exp (W_k).
\end{equation*}
Taking expectations of both sides we obtain
\begin{equation*}
  \e \exp(W_k)\le \e \exp(W).
\end{equation*}
A result of Harris \cite[Thm. 3.4]{harris} implies that the moment
generating function $\e \exp (\theta W)$, $\theta\ge 0$, of the
limit $W$ is finite everywhere, whenever the Galton-Watson process
is super-critical and has an offspring distribution with finite
support. This result yields that the right-hand side of the above
inequality is finite, and since it is independent of $k$, this
proves \eqref{eq:measure_non-detection}.
\end{proof}


\section{General detection and recovery}
\label{section:preliminaries} In this section we present conditions
for signal spaces to admit detection and recovery. We start by
describing operations on spaces which preserve the properties of
admitting detection or recovery. We then present explicit detection
and recovery algorithms applicable to a variety of signal spaces.

\subsection{New spaces from old}
Suppose $\X$ consists of just one signal $x$. It was already
mentioned in the introduction that if $x\in\ell^2$ then $\X$ does
not admit detection. This follows from Kakutani's theorem on
singularity of product measures \cite{kakutani} and is also a
consequence of our Theorem~\ref{thm:general_non_detection_intro}. In
the converse direction, if $x\notin\ell^2$ then $\X$ admits
detection. This can be deduced, for example, from
Theorem~\ref{thm:general_detection_intro} (to be proved later, see
Corollary~\ref{cor:vol_growth_detect} (i)). It is useful, however,
to describe an explicit detection map for this case.
Assume without loss of generality that $x\in\R^\N$. Declare that an
observed noisy signal $z$ is pure noise (that is, $z=\xi$) if and
only if
\begin{equation}\label{eq:one-point_space_detection_map}
  \limsup_{N\to\infty}\sum_{n=1}^N \left(x_n z_n - \frac{1}{2}x_n^2\right)\le 0.
\end{equation}
Observing that the random walk $(\sum_{n\le N} x_n \xi_n)_{N\ge 1}$
has the same distribution as $(B(\sum_{n\le N} x_n^2))_{N\ge 1}$,
where $B$ is a standard Brownian motion, the validity of this
detection map follows from the law of the iterated logarithm or the
law of large numbers for Brownian motion.

Now consider the case that $\X$ consists of exactly two signals
$x,y$. By subtracting $x$ from the observed signal we see that $\X$
admits recovery if and only if $\X':=\{y-x\}$ admits detection. Thus
we conclude that $\X$ admits recovery if and only if
$x-y\notin\ell^2$ and we may obtain an explicit recovery map by
replacing $x$ with $x-y$ in the condition
\eqref{eq:one-point_space_detection_map}. An immediate consequence
of the next proposition is that these results extend to countable
signal spaces. Precisely, a countable space $\X$ admits detection if
and only if $\X\cap\ell^2 = \emptyset$ and it admits recovery if and
only if $(\X - \X)\cap\ell^2 = \{0\}$.
\begin{proposition}\label{prop:union_of_spaces}
  Let $(\X_i)$, $i\in\N$, be signal spaces on the same index set and $\X =
  \cup \X_i$.
  \begin{enumerate-math}
    \item $\X$ admits detection if and only if each $\X_i$ admits detection.
    \item If $\X$ admits recovery then each of the spaces $\X_i$ admits
    recovery.
    \item Suppose that each of the spaces $\X_i$ is a Borel set.
    If each of the spaces $\X_i$ admits
    recovery and $\sum |x_n - y_n|^2 =
    \infty$ for every distinct $x,y\in \X$ then $\X$ admits recovery.
  \end{enumerate-math}
\end{proposition}
\begin{proof}
  Since $\X_i\subset\X$ for each $i$, it is clear that if $\X$
  admits detection or recovery then so does $\X_i$. Suppose now that
  each $\X_i$ admits detection and let $T_i$ be the corresponding
  detection function. It is simple to check that the mapping $T$
  defined by $T(z):=1 - \prod (1 - T_i(z))$ defines a detection
  function for $\X$.

  Now suppose that each $\X_i$ admits
  recovery and that $(\X-\X)\cap\ell^2 = \{0\}$. Let $T_i$ be the recovery
  function for each $\X_i$. Denote by $T_{x,y}$ the recovery
  function for the two-point space $\{x,y\}$ which arises from \eqref{eq:one-point_space_detection_map}. Observe that $T_{x,y}$ is Borel measurable jointly in its argument and in the
  pair $x,y$.
  Suppose that $z$ is the observed noisy signal from the space $\X$ and set
  $x_i:=T_i(z)$. Let $\mathcal{I}:=\{i\colon x_i\in\X_i\}$. It follows that, almost surely, $z = x_{i_0} + \xi$ for some $i_0\in\mathcal{I}$. By our assumptions, $x_i-x_j\notin\ell^2$ whenever $i,j\in\mathcal{I}$, $i\neq
  j$. Hence, almost surely, $i_0$ is the unique $i\in\mathcal{I}$ for
  which
  \begin{equation}\label{eq:recovery_in_union_cond}
  T_{x_i,x_j}(z) = x_i\text{ for all $j\in\mathcal{I}$, $j\neq i$}.
  \end{equation}
  Finally, we may obtain a recovery
  function $T$ for $\X$ by setting $T(z)$ to equal $x_i$ for the unique $i\in\mathcal{I}$ satisfying \eqref{eq:recovery_in_union_cond}, if such an $i$ exists, and setting $T(z)$ to be an arbitrary fixed
  sequence otherwise. The properties of the recovery functions $T_i$, $T_{x_i,x_j}$ and the assumption that each $\X_i$ is Borel ensure
  that $T$ is indeed a Borel measurable recovery function.
\end{proof}
We may use the above ideas to formalize the intuitive fact that
detection is easier than recovery.
\begin{proposition}\label{prop:recovery_implies_detection}
  Let $\X$ be a signal space satisfying $\X\cap\ell^2 = \emptyset$.
  If $\X$ admits recovery then $\X$ admits detection.
\end{proposition}
\begin{proof}
  Let $T$ be the recovery function for $\X$. Write $D_{x}$ for the detection function from the one-point space $\{x\}$ which arises from
  \eqref{eq:one-point_space_detection_map}. Declare that an observed
  signal $z$ is pure noise if either $T(z) \in \ell^2$ or
  $D_{T(z)}(z) = 0$. It is simple to check, using
  that $\ell^2$ is a Borel set, that this defines a
  valid detection function for $\X$.
\end{proof}

In our setup we take the same noise level for all coordinates. I.e.,
our noise is a standard Gaussian for all coordinates. Suppose
instead that we replace the noise $\xi_n$ in coordinate $n$ by
$\sigma_n\xi_n$ for some positive sequence $\{\sigma_n\}$. Upon
observing the noisy signal $\{x_n + \sigma_n \xi_n\}$, the receiver
may scale each coordinate and obtain the signal
$\{\frac{1}{\sigma_n} x_n + \xi_n\}$. Thus, changing the noise level
is equivalent to scaling the signal space $\X$. It seems intuitively
clear that reducing the noise level should make the tasks of
recovery and detection easier. The next proposition makes this
precise.
\begin{proposition}\label{prop_monotonicity_in_noise}
  Let $\X$ be a signal space and $\{c_n\}$ be a sequence such that $c_n\ge 1$ for all $n$. Let $\X_c:=\{\{c_n x_n\}\colon x\in\X\}$.
  \begin{enumerate-math}
    \item If $\X$ admits detection then $\X_c$ admits detection.
    \item If $\X$ admits recovery then $\X_c$ admits recovery.
  \end{enumerate-math}
\end{proposition}
\begin{proof}
  Suppose $\X$ admits recovery via the function $T$ and
  let us show that $\X_c$ admits recovery. Let $z$ be the observed
  noisy signal and suppose that $z_n = c_n x_n + \xi_n$ for some
  $x\in\X$. Let $\{\xi_n'\}$ be an independent noise sequence, i.e., a
  sequence of independent standard normal random variables,
  independent of $\xi$. Observe that the sequence $z'$ defined by
  \begin{equation}\label{eq:z_n_prime_def}
  z_n':=\frac{z_n + \sqrt{c_n^2-1}\cdot\xi_n'}{c_n}
  \end{equation}
  has the same distribution as $\{x_n +
  \xi_n\}$. Thus, almost surely (on the product space of $\xi$ and $\xi'$),
  $T(z') = x$. Hence we may define the recovery function $T'$ for
  $\X_c$ as follows. Given $z$, we consider the random sequence $z'$ defined by
  \eqref{eq:z_n_prime_def}. Then we set $T'(z)=x$ for $x$ defined by $x_n = \e(T(z')_n)$, where the expectation is over $\xi'$ (or we set $x_n$ to zero, say, if the expectation does not exist).
  The fact that $T'$ is a recovery function for $\X_c$ follows from the above arguments by Fubini's
  theorem.
  This establishes the second part of the proposition. The proof of the first part is analogous.
\end{proof}
As a corollary we deduce that there exists a critical noise level
for detection and recovery. Precisely, for any space $\X$ there
exists some $0\le c_d\le \infty$ such that if $0<c< c_d$ then
$c\X:=\{\{cx_n\}\colon x\in\X\}$ does not admit detection and if
$c_d<c<\infty$ then $c\X$ admits detection. An analogous threshold
exists for recovery from $c\X$.
Theorem~\ref{thm:trig_space_recovery_intro} and the example in
Section~\ref{sec:tree_trail} show that these thresholds may
sometimes be strictly between zero and infinity.

Another natural way to combine spaces is via a product construction.
Given two spaces $\X_1$ and $\X_2$ one may form the product space
\begin{equation*}
  \X_1 \times \X_2 := \{(x_1, x_2)\colon x_1\in\X_1,\, x_2\in\X_2\}
\end{equation*}
and consider the recovery and detection problems for it (under noise
of the form $(\xi_1, \xi_2)$ with $\xi_1$ and $\xi_2$ being
independent sequences of independent standard Gaussian random
variables).
\begin{proposition}\label{prop:product_spaces}
Let $\X_1$ and $\X_2$ be non-empty signal spaces.
\begin{enumerate-math}
  \item $\X_1\times\X_2$ admits detection if and only if at least one of $\X_1$
  and $\X_2$ admits detection.
  \item $\X_1\times\X_2$ admits recovery if and only if both $\X_1$
  and $\X_2$ admit recovery.
\end{enumerate-math}
\end{proposition}
\begin{proof}
  It is clear that if both $\X_1$ and $\X_2$ admit recovery then
  $\X_1\times\X_2$ admits recovery since we may apply the recovery
  mappings of $\X_1$ and $\X_2$ separately to the relevant
  coordinates of the signals in the product space. It is similarly
  clear that if at least one of $\X_1$ and $\X_2$ admits detection
  then also $\X_1\times\X_2$ admits detection.

  Now suppose that $\X_1\times\X_2$ admits recovery, via the mapping
  $T$, and let us prove that $\X_1$, say, admits recovery. Let $z_1$ be the observed
  noisy signal from the space $\X_1$ and write $z_1 = x_1 + \xi_1$ for some
  $x_1\in\X_1$ and a  noise sequence $\xi_1$. Fix some $x_2\in\X_2$. Let $\xi_2$ be a noise sequence, suitable for $\X_2$ and independent of
  $\xi_1$, and define $z_2 = x_2 + \xi_2$. Observe that, by Fubini's theorem and the properties of $T$,
  \begin{equation*}
    \text{almost surely with respect to $\xi_1$,}\quad\P(T(z_1, z_2) = (x_1, x_2)\, |\, z_1) =
    1.
  \end{equation*}
  Thus, in particular,
  \begin{equation*}
    \text{almost surely with respect to $\xi_1$,}\quad \e(T(z_1,
    z_2)\,|\,z_1) = (x_1, x_2)
  \end{equation*}
  yielding a recovery mapping for $\X_1$.

  Now suppose that $\X_1\times\X_2$ admits detection, via the mapping $T$. Let, again, $\xi_1$ and $\xi_2$ be independent noise sequences suitable for $\X_1$ and $\X_2$ respectively. Consider first
  the possibility that there exists some $x_1\in\X_1$ such that
  \begin{equation}\label{eq:detection_mapping_product_space_prop}
    \text{with positive probability with respect to
    $\xi_1$,}\quad\P(T(x_1 + \xi_1, \xi_2) = 0\,|\, \xi_1) = 1.
  \end{equation}
  We claim that, consequently, $\X_2$ admits detection. Indeed,
  suppose that $z_2$ is an observed noisy signal, where either $z_2
  = \xi_2$ or $z_2 = x_2 + \xi_2$ for some $x_2\in\X_2$. We may
  decide between these possibilities by declaring $z_2$ to be pure
  noise if and only if $\e(T(x_1 + \xi_1, z_2)\,|\,z_2) < 1$. It is
  simple to check, using \eqref{eq:detection_mapping_product_space_prop}, Fubini's theorem and the properties of $T$, that this defines a valid detection
  map for $\X_2$.

  Now suppose that \eqref{eq:detection_mapping_product_space_prop}
  is violated, which is to say that for every $x_1\in \X_1$,
  \begin{equation}\label{eq:detection_mapping_product_spaces_converse}
    \text{almost surely with respect to $\xi_1$,}\quad \P(T(x_1 + \xi_1, \xi_2) = 1\,|\, \xi_1) > 0.
  \end{equation}
  We claim that, consequently, $\X_1$ admits detection. Indeed,
  suppose now that $z_1$ is an observed noisy signal, where either $z_1
  = \xi_1$ or $z_1 = x_1 + \xi_1$ for some $x_1\in\X_1$. We may
  decide between these possibilities by declaring $z_1$ to be pure
  noise if and only if $\e(T(z_1, \xi_2)\,|\,z_1) = 0$.
  Again, \eqref{eq:detection_mapping_product_spaces_converse}, Fubini's theorem and the properties of $T$ ensure that this defines a valid detection
  map for $\X_1$.
\end{proof}
We remark that using an analogous proof one may show that if $\X$ is
a product of \emph{countably} many non-empty signal spaces $\X_i$
then $\X$ admits recovery if and only if each of the $\X_i$ admits
recovery. However, it is no longer true that if $\X$ admits
detection then at least one of the $\X_i$ admits detection. For
instance, if $\X_i = \{x_i\}$ for a single signal $x_i\in\ell^2$
then each of the $\X_i$ does not admit detection (since the
necessary condition \eqref{eq:detection_nec_cond} is violated), but
noting that $\X$ contains only a single signal we see that $\X$
admits detection if $\sum_i \|x_i\|_2^2 = \infty$.

\subsection{General detection and recovery
algorithms}\label{sec:general_detection_recovery_algorithms} In this
section we propose general algorithms for detection and recovery and
study conditions for them to apply to a given space $\X$.

Our algorithms depend on the following input: For each $k\ge 1$, let
$f_k:\R^\N\times\R^\N\to\R$ be a continuous function with respect to
the product topology on $\R^\N$. One may think of $f_k(x,z)$ as a
measure of the similarity of $x$ and $z$, assigning a ``score'' to
the possibility that $z$ is a noisy version of $x$, with this
``score'' taking into account more and more information about $x$
and $z$ as $k$ increases. A concrete example of such functions is
given by \eqref{eq:f_k_concrete_choice} below.

We denote by $\open_x$ the set of open sets in the product topology
on $\R^\N$ which contain a given point $x$.

We now describe the general recovery algorithm from a space $\X$.
Given a noisy signal $z\in\R^\N$ we say that $x\in\X$ is a
\emph{candidate} if
\begin{equation}\label{eq:candidate_def}
  \forall U\in\open_x:\quad\;f_k(x,z)>\sup_{y\in\X\setminus U} f_k(y,z)\;\;\text{ for all
  but finitely many $k$}.
\end{equation}
We note that the number of candidates is necessarily either zero or
one. We define the recovery mapping $T$ by $T(z)=x$ if $x$ is the
(unique) candidate, and by $T(z)=0$ if there are no candidates.

We mention the useful fact that if $d$ is an arbitrary metric on
$\R^\N$ generating the product topology then $x\in\X$ is a candidate
if and only if
\begin{equation*}
  \forall \eps>0:\quad\;f_k(x,z)>\sup_{\substack{y\in\X\\d(x,y)>\eps}} f_k(y,z)\;\;\text{ for all
  but finitely many $k$}.
\end{equation*}
\begin{lemma}\label{lem:general_recovery_measurable}
  Suppose that $\X$ is a Borel subset of $\R^\N$. Then the mapping $T$ defined above is Borel measurable.
\end{lemma}
\begin{proof}
  Fix an arbitrary metric $d$ generating the product topology on $\R^\N$. Without loss of generality, we may assume that $f_k$ takes values in $(0,1)$. For each $y\in\R^\N$ and $\eps>0$, define a
  function $g_{k,y,\eps}:\R^\N\times\R^\N\to\R$ by
  \begin{equation*}
    g_{k,y,\eps}(x,z):=f_k(y,z)\cdot \1_{A(y,\eps)}(x),
  \end{equation*}
  where $A(y,\eps):=\{x\in\R^\N\,:\, d(x,y)>\eps\}$. Observe that
  $g_{k,y,\eps}$ is a lower semi-continuous function (since $f_k$ are continuous and $A_{y,\eps}$ is open). Hence the
  function
  \begin{equation*}
    h_{k,\eps}(x,z):=\sup_{y\in\X} g_{k,y,\eps}(x,z)
  \end{equation*}
  is also lower semi-continuous, and in particular Borel. Define the
  sets
  \begin{equation*}
    \Omega:=\bigcap_{j\ge 1} \bigcup_{k_0\ge 1} \bigcap_{k\ge k_0} \Omega_{j,k},\quad \Omega_{j,k}:=\{(x,z)\in\R^\N\times\R^\N\,:\,
    f_k(x,z)>h_{k,1/j}(x,z)\},
  \end{equation*}
  which are Borel by the above. Observe that $(x,z)\in\Omega_{j,k}$
  if and only if
  \begin{equation*}
    f_k(x,z)>\sup_{\substack{y\in\X\\d(x,y)>1/j}} f_k(y,z).
  \end{equation*}
  Hence $x$ is a candidate for a noisy signal $z$ if and only if
  $x\in\X$ and $(x,z)\in\Omega$.

  Finally, to see that $T$ is Borel it suffices to show that for any Borel set
  $A\subset\R^\N\setminus\{0\}$ the set
  \begin{equation*}
    T^{-1}(A) = P(((A\cap\X)\times \R^\N)\cap\Omega)
  \end{equation*}
  is Borel, where $P$ is the projection mapping $P(x,z)=z$. This
  follows from a theorem of descriptive set theory (see, e.g., \cite[Theorem 15.1]{kechris}) which states that if $B$ is a Borel subset of $\R^\N\times\R^\N$
  such that $P$ is one-to-one on $B$ then $P(B)$ is Borel.
\end{proof}
For a space $\X$ to admit recovery via the above algorithm, it is
necessary and sufficient that for every $x\in\X$ and every
$U\in\open_x$, almost surely,
\begin{equation}\label{eq:f_k_general_cond}
  f_k(x,x+\xi)>\sup_{y\in\X\setminus U} f_k(y,x+\xi)\;\;\text{ for all
  but finitely many $k$}.
\end{equation}

To obtain a concrete test, we need to specify the functions $f_k$.
Since the $f_k$ measure the similarity of $x$ and $z$ ``at level
$k$'', a natural choice is
\begin{equation}\label{eq:f_k_concrete_choice}
  f_k(x,z):=-\sum_{n=1}^k (z_n-x_n)^2.
\end{equation}
The functions $f_k$ are essentially the log-likelihood of seeing the
first $k$ coordinates of the noisy signal $z$ when the original
signal is $x$. It is also sometimes useful in applications to
consider a fixed subsequence of $k$'s. This yields the following
statement.
\begin{theorem}\label{thm:recover_likelihood2}
  Let $\X$ be a Borel subset of $\R^\N$, and $\{k_j\}$ be an
increasing sequence tending to infinity. Suppose that for every
  $x\in\X$ and every $U\in\open_x$, almost surely,
  \begin{equation}\label{eq:recover_likelihood_cond2}
    \inf_{y\in\X\setminus U}\; \sum_{n=1}^{k_j} \left[(y_n-x_n)^2 -
    2\xi_n(y_n-x_n)\right] > 0\;\;\text{ for all
  but finitely many $j$}.
  \end{equation}
  Then $\X$ admits recovery (via the above recovery mapping).
\end{theorem}

We now describe our general detection algorithm from a space $\X$,
which relies on similar principles. Given a noisy signal $z\in\R^\N$
we declare that it is pure noise if
\begin{equation}\label{eq:general_detection_test}
  \forall U\in\open_0:\quad\;f_k(0,z)>\sup_{y\in\X\setminus U} f_k(y,z)\;\;\text{ for infinitely many $k$}.
\end{equation}
Otherwise we declare that $z$ is a noisy version of a signal from
$\X$. The fact that this detection mapping is Borel measurable,
whenever $\X$ is Borel, follows similarly to the proof of
Lemma~\ref{lem:general_recovery_measurable} above.

Let us specialize to the choice \eqref{eq:f_k_concrete_choice} of
the functions $(f_k)$. We have the following statement.
\begin{theorem}\label{thm:detect_likelihood2}
  Let $\X$ be a Borel subset of $\R^\N\setminus\{0\}$. Suppose that for every $U\in\open_0$, almost surely,
  \begin{equation}\label{eq:detect_likelihood_cond2}
    \inf_{y\in\X\setminus U}\; \sum_{n=1}^{k} \left[y_n^2 -
    2\xi_n y_n\right] > 0\;\;\text{ for infinitely many $k$}.
  \end{equation}
  Then $\X$ admits detection (via the above detection mapping).
\end{theorem}
\begin{proof}
  Suppose first that $z=\xi$ is pure noise. It is simple to check that
  condition \eqref{eq:general_detection_test} is equivalent to
  condition \eqref{eq:detect_likelihood_cond2} for the choice
  \eqref{eq:f_k_concrete_choice} of the functions $(f_k)$. Thus the
  above detection algorithm will declare $z$ to be pure noise.

  Next, we show that $\X\cap\ell^2 = \emptyset$. Indeed, suppose
  that $y\in \X\cap\ell^2$. Then $\sum_{n=1}^\infty \xi_n y_n$
  converges almost surely to a centered Gaussian random variable.
  Moreover, with positive probability,
  \begin{equation*}
    2\sum_{n=1}^\infty \xi_n y_n > \sum_{n=1}^\infty y_n^2,
  \end{equation*}
  which contradicts our assumption that
  \eqref{eq:detect_likelihood_cond2} holds with probability one.

  Finally, suppose that $z = y+\xi$ for some $y\in\X$. Let us
  show that, almost surely,
  \begin{equation}\label{eq:zero_y_comparison}
    f_k(0,z)<f_k(y,z)\;\;\text{ for all but finitely many $k$}.
  \end{equation}
  This will imply that the detection algorithm will declare $z$ to
  be a noisy version of a signal. Substituting the definition of $z$
  and the $(f_k)$ into \eqref{eq:zero_y_comparison} yields the
  equivalent condition
  \begin{equation}\label{eq:random_walk_slln_cond}
    \sum_{n=1}^k y_n^2>-2\sum_{n=1}^k \xi_n y_n\;\;\text{ for all but finitely many $k$}.
  \end{equation}
  To verify this condition let $B(t)$ be a standard Brownian motion
  and observe that the random walk $(\sum_{n=1}^k \xi_n
  y_n)_k$ has the same distribution as $(B(\sum_{n=1}^k
  y_n^2))_k$. Thus, since $\sum y_n^2 = \infty$,
  condition \eqref{eq:random_walk_slln_cond} follows from the law of
  the iterated logarithm or the law of large numbers for Brownian
  motion.
\end{proof}

The above general tests connect the notions of recovery and
detection with the general theory of extrema of Gaussian processes.
This allows to give conditions for detection and recovery using
notions from the general theory, such as metric entropy, the Dudley
integral and the generic chaining (see, e.g., \cite{talagrand}). We
illustrate this possibility in its simplest manifestation, for
spaces with only finitely many possible values in each coordinate.
For such spaces one can give very simple, albeit imprecise,
conditions for detection and recovery based on volume growth.
Denote by $\proj_k(\X)$ the projection of the space $\X$ to the
first $k$ coordinates, i.e.,
\begin{equation}\label{eq:proj_def}
 \proj_k(\X):=\{(x_1,\ldots, x_k)\,:\, x\in\X\}.
\end{equation}
\begin{corollary}\label{cor:vol_growth_detect}
  Let $\X$ be a Borel subset of $\R^\N$ or $\C^\N$.
\begin{enumerate-math}
\item  If $\X\cap\ell^2=\emptyset$ and
\begin{equation}\label{eq:discrete_entropy_cond}
  \liminf_{k\to\infty} \frac{\log|\proj_k(\X)|}{\inf_{y\in\X}\sum_{n=1}^k |y_n|^2} <
  \frac{1}{8}
\end{equation}
then $\X$ admits detection.
\item
Let $\{k_j\}$ be an increasing sequence tending to infinity. Let
\begin{equation*}
\X_i^x:=\{y\in \X\,:\, (y_1,\ldots, y_i)\neq (x_1,\ldots, x_i)\}.
\end{equation*}
Suppose that for every $x\in\X$ and every $i\ge 1$,
\begin{equation}\label{eq:discrete_recover_entropy_cond}
  \sum_j |\proj_{k_j}(\X_i^x)|\exp\bigg(-\frac{1}{8} \inf_{y\in\X_i^x}\sum_{n=1}^{k_j} |y_n-x_n|^2\bigg)
  < \infty.
\end{equation}
Then $\X$ admits recovery.
\end{enumerate-math}
\end{corollary}
We remark that the conditions in the corollary are sharp, in the
sense that the constant $\frac{1}{8}$ cannot be raised to an
arbitrarily large value. This can be seen from the results in
Section~\ref{sec:unknown_phase} or Section~\ref{sec:tree_trail}. We
remark also that the first part of the corollary remains true (with
an analogous proof) if one changes
condition~\eqref{eq:discrete_entropy_cond} into the weaker condition
that for every $U\in\open_0$,
\begin{equation*}
  \liminf_{k\to\infty} \frac{\log|\proj_k(\X)|}{\inf_{y\in\X\setminus U}\sum_{n=1}^k |y_n|^2} <
  \frac{1}{8}.
\end{equation*}
However, this modified statement can in fact be deduced from the
current statement by using the decomposition $\X:=\cup
\left(\X\setminus U\right)$, where the union is over $U$ in a
countable basis for $\open_0$, and applying
Proposition~\ref{prop:union_of_spaces}.
\begin{proof}[Proof of Corollary~\ref{cor:vol_growth_detect}]
(i) The case that $\X\subseteq\C^\N$ follows from the case that
$\X\subseteq\R^\N$ by writing each complex number as two reals. Thus
we suppose that $\X\subseteq\R^\N$. We claim that the assumptions
imply that
\begin{equation}\label{eq:detect_cor_cond}
  \text{with probability one, }\,\sup_{y\in\X}\, \frac{\sum_{n=1}^{k} y_n\xi_n}{\sum_{n=1}^{k}
  y_n^2}< \frac{1}{2}\,
  \text{ for infinitely many
  $k$},
\end{equation}
This will imply \eqref{eq:detect_likelihood_cond2} due to the fact
that for any $U\in\open_0$,
\begin{equation*}
\inf_{y\in\R^\N\setminus U}\;\sum_{n=1}^{k}y_n^2>0\quad\text{for all
sufficiently large $k$}.
\end{equation*}
Thus, Theorem~\ref{thm:detect_likelihood2} shows that $\X$ admits
detection.

Define the events
\begin{equation*}
\Omega_k := \Bigg\{\exists y\in\X,\, \sum_{n=1}^{k} y_n\xi_n\ge
  \frac{1}{2}\sum_{n=1}^{k} y_n^2\Bigg\}.
\end{equation*}
The fact that $\P(\liminf_{k\to\infty} \Omega_k) \le
\liminf_{k\to\infty} \P(\Omega_k)$ together with our assumption that
$|\proj_k(\X)|<\infty$ for all $k$ will imply
\eqref{eq:detect_cor_cond} once we show that
\begin{equation}\label{eq:Omega_k_prob_liminf}
  \liminf_{k\to\infty}\, \P(\Omega_k) = 0.
\end{equation}
To see \eqref{eq:Omega_k_prob_liminf}, denote
\begin{equation*}
   \sigma_k^2 :=
\inf_{y\in\X}\, \sum_{n=1}^k y_n^2
\end{equation*}
and observe first that
\begin{equation}\label{eq:limsup_sigma_k_log_proj}
  \limsup_{k\to\infty}\,\frac{1}{8}\sigma_k^2 - \log|\proj_k(\X)| =
  \infty.
\end{equation}
This is clear if $\X$ is finite since $\X\cap\ell^2=\emptyset$, and
follows also if $\X$ is infinite by
\eqref{eq:discrete_entropy_cond}. Denoting by $G$ a standard normal
random variable we have
\begin{align*}
  \P(\Omega_k)&\le \sum_{y\in\proj_{k}(\X)}\P\left(\sum_{n=1}^{k} y_n\xi_n\ge
  \frac{1}{2}\sum_{n=1}^{k} y_n^2\right) = \sum_{y\in\proj_{k}(\X)} \P\left(G\ge
  \frac{1}{2}\Big(\sum_{n=1}^{k} y_n^2\Big)^{1/2}\right)\\
  &\le |\proj_{k}(\X)|\,\P\left(G\ge
  \frac{1}{2}\sigma_{k}\right).
\end{align*}
Thus, by standard estimates for normal random variables,
\begin{equation}\label{eq:Omega_k_upper_bound}
  \P(\Omega_k) \le c |\proj_{k}(\X)|e^{-\frac{1}{8}\sigma_{k}^2}
\end{equation}
for some universal constant $c>0$. Hence
\eqref{eq:limsup_sigma_k_log_proj} implies
\eqref{eq:Omega_k_prob_liminf}, from which the first part of the
corollary follows.

(ii) We may again suppose WLOG that $\X\subseteq\R^\N$. We claim
that the assumptions imply that for every $x\in\X$ and every $i$,
\begin{equation}\label{eq:recover_cor_cond}
  \text{with probability one, }\,\sup_{y\in\X_i^x}\, \frac{\sum_{n=1}^{k_j} (y_n-x_n)\xi_n}{\sum_{n=1}^{k_j}
  (y_n-x_n)^2}< \frac{1}{2}\,
  \text{ for all but finitely many
  $j$}.
\end{equation}
This will imply \eqref{eq:recover_likelihood_cond2} due to the fact
that for any $U\in\open_x$, $\X_i^x\subseteq \X\setminus U$ for some
$i=i(x,U)$.
Thus, Theorem~\ref{thm:recover_likelihood2} shows that $\X$ admits
recovery.

Now, to see that \eqref{eq:recover_cor_cond} holds, observe that if
we define
\begin{equation*}
  \Omega_k^{x,i} := \Bigg\{\exists y\in\X_i^x,\, \sum_{n=1}^{k} (y_n-x_n)\xi_n\ge
  \frac{1}{2}\sum_{n=1}^{k} (y_n-x_n)^2\Bigg\}
\end{equation*}
and
\begin{equation*}
  (\sigma_k^{x,i})^2 := \inf_{y\in\X_i^x}\, \sum_{n=1}^{k} (y_n-x_n)^2,
\end{equation*}
then, just as in the proof of \eqref{eq:Omega_k_upper_bound}, we
have
\begin{equation*}
  \P(\Omega_k^{x,\eps})\le c |\proj_{k}(\X_i^x)|e^{-\frac{1}{8}(\sigma_{k}^{x,i})^2}
\end{equation*}
for some universal constant $c>0$. Thus, \eqref{eq:recover_cor_cond}
follows from \eqref{eq:discrete_recover_entropy_cond} and the
Borel-Cantelli lemma.
\end{proof}

We end the section with two remarks.

First, we briefly comment on the optimality of the above tests. It
is natural to ask whether every space $\X$ admitting recovery or
detection satisfies the conditions of
Theorems~\ref{thm:recover_likelihood2} or
\ref{thm:detect_likelihood2}, respectively. This is not the case, at
least for the detection problem. Indeed,
Theorem~\ref{thm:Rademacher_detection} (to be proved later, see
Proposition~\ref{sense_prop}) shows that if $\sigma=\{\sigma_n\}$ is
a sequence of positive numbers satisfying $\sum \sigma_n^4 = \infty$
then the ``Rademacher space'',
\[ \X(\sigma)=\left\{x : x_n = \pm\sigma_n\text{ for all
$n$}\right\}
\]
admits detection. Now note that for every possible noise $\xi$ there
exists a $y\in\X(\sigma)$ for which the sign of $y_n$ equals the
sign of $\xi_n$ for all $n$. Plugging this $y$ in condition
\eqref{eq:detect_likelihood_cond2} of
Theorem~\ref{thm:detect_likelihood2} shows that if the theorem is to
apply to $\X(\sigma)$ then, in particular, almost surely,
\begin{equation}\label{eq:detect_cond_counterexample}
  \sum_{n=1}^{k} \sigma_n|\xi_n|<
  \frac{1}{2}\sum_{n=1}^{k} \sigma_n^2\text{ for infinitely many $k$}.
\end{equation}
However, it is simple to check that if $\sigma_n\to0$ and
$\sum_{n=1}^\infty \sigma_n=\infty$ then
\eqref{eq:detect_cond_counterexample} is almost-surely violated.
Thus, there exist spaces admitting detection to which
Theorem~\ref{thm:detect_likelihood2} does not apply. It may be worth
mentioning though, that the detection test used in the proof of
Theorem~\ref{thm:Rademacher_detection} still uses the notion of
likelihood but in a different manner than in
Theorem~\ref{thm:detect_likelihood2}, see the remark after the proof
of Proposition~\ref{sense_prop}. We expect that
Theorem~\ref{thm:recover_likelihood2} is similarly non-optimal,
though we do not present an example to this end. Nevertheless, the
above detection and recovery tests will prove useful in several of
our subsequent examples.

Second, we mention a possible variant of the above tests. As
explained in the beginning of the section, our approach to the
recovery problem is based on identifying a candidate for the
transmitted signal using the condition~\eqref{eq:candidate_def}. The
following is an alternative condition. Given a noisy signal
$z\in\R^\N$ we say that $x\in\X$ is a \emph{candidate} if
\begin{equation}\label{eq:candidate_def2}
  \forall y\in\X\setminus\{x\}:\quad\;f_k(x,z)>f_k(y,z)\;\;\text{ for all
  but finitely many $k$}.
\end{equation}
Again, we have the property that the number of candidates is
necessarily either zero or one and we may define the recovery
mapping $T$ by $T(z)=x$ if $x$ is the (unique) candidate, and by
$T(z)=0$ if there are no candidates. This definition may seem more
natural than \eqref{eq:candidate_def} as it avoids the use of
topology (except possibly in the choice of the $f_k$) and since it
is formally stronger than \eqref{eq:candidate_def} in the sense that
if $x$ is a candidate according to \eqref{eq:candidate_def} then it
is also a candidate according to \eqref{eq:candidate_def2}. However,
the main disadvantage of this definition is that it is not clear
whether the resulting recovery mapping $T$ is Borel measurable or
even universally measurable, even when specializing to the choice
\eqref{eq:f_k_concrete_choice} and assuming that $\X$ is Borel.

A detection mapping of a similar nature may be defined as follows.
Given a noisy signal $z\in\R^\N$ declare that it is pure noise if
\begin{equation*}
  \forall y\in\X:\quad\;f_k(0,z)>f_k(y,z)\;\;\text{ for infinitely many $k$}.
\end{equation*}
This mapping has similar advantages and disadvantages as the above
recovery mapping except that, while it is not clear whether it is
Borel measurable, it is not difficult to check at least that the set
of $z$ which are not declared pure noise by it is an analytic set.

\section{Spaces with unknown phase}
\label{sec:unknown_phase}

\subsection{} In this section we consider signal spaces in which
the amplitudes of the signals are known to the receiver. The
amplitudes will be given by a sequence $\{\sigma_n\}$, given in
advance, and the signals differ solely in their signs, or phases in
the complex case. It is natural to let $\{\sigma_n\}$ be positive
numbers but the discussion below remains true for any complex
numbers.

First, we discuss two representative examples:
\begin{enumerate-math}
  \item The ``Rademacher space'' consisting of signals $x$ of the form
  $x_n = \pm \sigma_n$, with all possible choices of signs
  allowed.
  \item The ``Walsh space'', also consisting of signals $x$ of the form $x_n = \pm \sigma_n$,
  but allowing only a restricted set of possibilities for the signs.
  The signals are parametrized by
  all sequences $\eps=(\eps_j)_{j\ge 0}$ with $\eps_j=\pm 1$. To define $x(\eps)$ we
represent each integer $n\ge 0$ with its binary expansion
\begin{equation}
\label{eq:n-representation} n = \sum_{j \geq 0} n_j 2^j, \quad n_j
\in \{0,1\}, \quad \sum n_j < \infty,
\end{equation}
and let $x_n(\eps) = \sigma_n \prod \eps_j$ where the product is
over all $j$ such that $n_j\neq 0$.
\end{enumerate-math}

The following theorem shows that the conditions for detection and
recovery are rather different in these two examples.

\begin{theorem}\label{thm:detection_phases}\quad
\begin{enumerate-math}
  \item Detection in the Rademacher space is possible if and only if
  \begin{equation}\label{eq:Rademacher_detection_cond}
    \sum |\sigma_n|^4 = \infty,
  \end{equation}
  while recovery is not possible unless the $\{\sigma_n\}$ are all zero.
  \item Suppose the absolute values $\{|\sigma_n|\}$ are non-increasing. There exist absolute constants $0<a<b<\infty$ such that, if
  \begin{equation}\label{eq:cond_non_detect_Walsh}
    \limsup_{k\to\infty} \frac1{\log k} \sum_{n=0}^{k-1} |\sigma_n|^2
  \end{equation}
  is larger than $b$ then recovery is possible in the Walsh
  space, while if this limit is smaller than $a$
  then detection is impossible.
\end{enumerate-math}
\end{theorem}
The theorem and the remarks after
Proposition~\ref{prop_monotonicity_in_noise} imply the existence of
a space of signals $\X$ with the property that $\delta\X$ admits
recovery for large $\delta$ and does not admit detection for small
$\delta$, e.g., one may take the Walsh space with
$\sigma_n=1/\sqrt{n+1}$. Another example of this transition, with an
explicit calculation of the critical $\delta$ and a broader
discussion, is presented in Section~\ref{sec:tree_trail}. Such a
threshold phenomenon is not available in the Rademacher space.

The result about detection in  the Rademacher space extends to the
larger space of signals $\left\{x : |x_n|\ge \sigma_n\text{ for all
$n$}\right\}$, where in this case we take $\sigma_n\ge 0$. This is
proved in Proposition~\ref{sense_prop} below. It is interesting to
compare this also with the space $\left\{x : x_n\ge \sigma_n\text{
for all $n$}\right\}$, in which there is no absolute value on $x_n$,
where the condition for detection is $\sum \sigma_n^2 = \infty$. An
explicit detection function for this case arises by replacing
$\{x_n\}$ with $\{\sigma_n\}$ in
\eqref{eq:one-point_space_detection_map}.

Finally, observe that in a space with known amplitudes, a necessary
condition for both detection and recovery is $\sum
|\sigma_n|^2=\infty$, and that under this condition recovery implies
detection by Proposition~\ref{prop:recovery_implies_detection}.

\subsection{} The above examples can be seen as special cases of the
following setup. Let $\{\varphi_n\}$ be a sequence of functions
defined on a set $\Omega$. We consider the space consisting of
signals $\{x(\omega)\}$, parametrized by $\omega\in\Omega$, defined
by
\begin{equation}\label{eq:x_n_general}
x_n(\omega) = \sigma_n\, \varphi_n(\omega).
\end{equation}
In our examples, the $\varphi_n$ have modulus one and hence play the
role of phases. However, this restriction is sometimes relaxed
below.

The Rademacher space corresponds to the case where $\{\varphi_n\}$
are the Rademacher functions $\{r_n\}$, $n\ge 0$, defined as the
coordinate functions on the product space $\Omega=\{-1,1\}^\infty$.
The Walsh space is obtained by taking $\{\varphi_n\}$ to be the
Walsh functions $\{w_n\}$, $n\ge 0$, defined by $w_n(\omega)=\prod
r_j(\omega)$, where the product is over those $j$ such that the
$n_j$ in the expansion \eqref{eq:n-representation} are non-zero.

We consider two additional examples which form trigonometric
counterparts to the above spaces:
\begin{enumerate-math}
  \item The ``trigonometric space'' obtained by taking the trigonometric
  system $\varphi_n(t) = e^{2\pi int}$, $n\in\Z$, on the circle
  $\T$.
  \item The ``lacunary space'' obtained by taking
  $\varphi_n(t) = e^{2\pi i2^nt}$, $n\ge 0$, a lacunary subsequence of the trigonometric system.
\end{enumerate-math}

The trigonometric space is, in a sense, a counterpart to the Walsh
space. Indeed, in both systems $\{\varphi_n\}$ are the continuous
characters of a compact group, the Walsh functions for the Cantor
group $\{-1,1\}^\infty$ and the trigonometric functions for the
circle group $\T$. Similarly, the lacunary space is a trigonometric
counterpart to the Rademacher space. It is obtained as a lacunary
subsequence of the trigonometric system, whereas the Rademacher
functions can be seen as a lacunary subsequence of the Walsh system,
$r_n = w_{2^n}$.

One motivation for considering these spaces comes from the following
interpretation. Suppose that $\{\sigma_n\}$ are the Fourier
coefficients of an object $\mu$ on the circle $\T$ (say, a function
or a measure). Then the trigonometric space consists exactly of the
Fourier transforms of the rotations of $\mu$ by all possible angles.
Thus, the receiver seeks to recover the unknown rotation from the
noisy signal. The Walsh space example admits a similar
interpretation as Walsh transforms of the translates of an object,
however, not on the circle group $\T$ but rather on the Cantor group
$\{-1,1\}^\infty$.

The results stated in Theorem~\ref{thm:detection_phases} are true
also for the trigonometric and lacunary spaces. Precisely, the
theorem is true exactly as stated with the words ``Rademacher''
replaced by ``lacunary'', and ``Walsh'' replaced by
``trigonometric''. In addition, in the trigonometric case, since
signals are indexed by $\Z$, we need to take the sum in
\eqref{eq:cond_non_detect_Walsh} running over $|n|< k$, and to
interpret the non-increasing condition as meaning that
$|\sigma_{n+1}|\le|\sigma_n|$ for $n\ge 0$ and
$|\sigma_{n-1}|\le|\sigma_n|$ for $n\le 0$.

The rest of the section is devoted to the proofs. In fact, some of
the results hold in greater generality, as seen below.

\subsection{} We start by establishing the non-detection results
for the Rademacher and lacunary spaces under the assumption $\sum
|\sigma_n|^4<\infty$. In both cases we use the general non-detection
condition of Theorem~\ref{thm:general_detection}.

For the Rademacher space, we let $\P$ be the uniform measure on
$\Omega=\{-1,1\}^\infty$. Let $\omega,\omega'\in\Omega$ be sampled
independently according to $\P$. By
Theorem~\ref{thm:general_detection}, it suffices to show that
\begin{equation}\label{eq:Rademacher_non_detect_cond}
  \liminf_{k \to\infty} \; \e \, \exp \bigg\{\sum_{n=1}^k |\sigma_n|^2 r_n(\omega) r_n(\omega') \bigg\} <
  \infty,
\end{equation}
where the $\{r_n\}$ are the Rademacher functions. Since
$\{r_n(\omega)r_n(\omega'\}$ are independent random variables, each
distributed uniformly on $\{-1,1\}$, the above expression equals
\begin{equation}\label{eq:sigma_moment_calc}
  \liminf_{k \to\infty} \prod_{n=1}^k \left(\frac{1}{2}\exp(\sigma_n^2) +
  \frac{1}{2}\exp(-\sigma_n^2)\right).
\end{equation}
Thus, since $\frac{1}{2}(e^x+e^{-x})\le e^{x^2/2}$, $x\in\R$, and
$\sum |\sigma_n|^4 < \infty$, we deduce
\eqref{eq:Rademacher_non_detect_cond}.

For the lacunary space, letting $\P$ be the Lebesgue measure on
$\T$, the condition of Theorem~\ref{thm:general_detection} becomes
\begin{equation}\label{eq:lacunary_non_detect_cond}
  \liminf_{k\to\infty} \; \int_\T \, \exp \big(\Re f_k(t)\big)dt <
  \infty,
\end{equation}
where
\begin{equation*}
f_k(t):=\sum_{n=1}^k |\sigma_n|^2 e^{2\pi i2^n t}.
\end{equation*}
The assumption $\sum |\sigma_n|^4 < \infty$ implies, according to
\cite[Section V.8, Theorem 8.20]{zygmund}, the existence of $\mu>0$
small, not depending on $k$, such that
\begin{equation*}
  \int_{\T}\,\exp \big(\mu \Re(f_k(t))^2\big)dt \le C
\end{equation*}
with $C$ an absolute constant. This implies
condition~\eqref{eq:lacunary_non_detect_cond}.

\subsection{}
In this section we prove a positive detection result, which applies
in particular to the Rademacher and lacunary spaces. It is clear
that recovery in these spaces is not possible since the spaces
contain distinct elements whose difference is finitely supported
(unless the $\sigma_n$ are all zero).
\begin{proposition}\label{sense_prop}
Suppose the $\sigma_n$ are non-negative. Then detection in the space
\[ \X=\left\{x : |x_n|\ge \sigma_n\text{ for all
$n$}\right\}
\]
is possible if and only if $\sum \sigma_n^4 = \infty$.
\end{proposition}
\begin{proof}
The ``only if'' part follows from the non-detection result of the
previous section for the Rademacher (or lacunary) space. It remains
to prove the ``if'' part.

We may assume that $\sigma_n\le 1$ for all $n$, since otherwise we
can replace $\sigma_n$ by $\min(\sigma_n,1)$, thereby enlarging the
space of signals. Below we consider, without loss of generality, the
case that $\X$ is a space of complex-valued signals. One should keep
in mind that this means taking complex-valued noise, i.e., the real
and imaginary part of each $\xi_n$ are standard Gaussian random
variables. Finally, we assume that the signals are indexed by the
natural numbers.

The proof proceeds by analyzing an explicit detection algorithm. Let
$(N_k)$, $k\ge 1$, be a strictly increasing sequence satisfying
\begin{equation} \label{N_k_choice}
\sum_{k=1}^\infty \Bigg(\sum_{n=1}^{N_k}
\sigma_n^4\Bigg)^{-1}<\infty.
\end{equation}
For a sequence $y$, define
\begin{equation*}
S_N(y):=\sum_{n=1}^N \sigma_n^2
\left(|y_n|^2-(2+\tfrac{1}{2}\sigma_n^2)\right).
\end{equation*}
We define the detection map $T$ by setting $T(y)=1$ if
$S_{N_k}(y)\ge 0$ for infinitely many $k$, and $T(y)=0$ otherwise.
The map $T$ is clearly Borel measurable. Our goal is to show that
$T(\{\xi_n\})=0$ almost surely and that for every $x\in \X$,
$T(\{x_n+\xi_n\})=1$ almost surely.

We start with the pure noise case. We have
\begin{equation*}
\begin{split}
\e S_N(\{\xi_n\}) &= -\frac{1}{2}\sum_{n=1}^N \sigma_n^4,\\
\var(S_N(\{\xi_n\})) &= \sum_{n=1}^N \sigma_n^4 \var(|\xi_n|^2) = 4
\sum_{n=1}^N \sigma_n^4.
\end{split}
\end{equation*}
Define the events $A_k:=\{S_{N_k}(\{\xi_n\})\ge 0\}$ for $k\ge 1$.
By Chebyshev's inequality and our assumption that $0\le \sigma_n\le
1$, we have
\begin{equation*}
\P(A_k)\le \frac{\var(S_{N_k}(\{\xi_n\}))}{\left(\e
S_{N_k}(\{\xi_n\})\right)^2} = \frac{16\sum_{n=1}^{N_k}
\sigma_n^4}{\left(\sum_{n=1}^{N_k} \sigma_n^4\right)^2} =
\frac{16}{\sum_{n=1}^{N_k} \sigma_n^4}.
\end{equation*}
Applying \eqref{N_k_choice} we deduce that $\sum_k \P(A_k)<\infty$
and hence only finitely many of the $A_k$ occur almost surely. Thus
$T(\{\xi_n\})=0$ almost surely, as required.

We turn next to the noisy signal case. Fix $x\in \X$. We have
\begin{equation*}
\begin{split}
\e S_N(\{x_n + \xi_n\}) &= \sum_{n=1}^N
\sigma_n^2(|x_n|^2-\frac{1}{2}\sigma_n^2) \ge
\frac{1}{4}\sum_{n=1}^N \sigma_n^2(\sigma_n^2+|x_n|^2),\\
\var(S_N(\{x_n + \xi_n\})) &= \sum_{n=1}^N \sigma_n^4
\var(|x_n+\xi_n|^2)
=\\
&= 4 \sum_{n=1}^N \sigma_n^4(1+|x_n|^2) \le 4\sum_{n=1}^N
\sigma_n^2(\sigma_n^2+|x_n|^2),
\end{split}
\end{equation*}
where in the last inequality we have used our assumption that $0\le
\sigma_n\le 1$. Define the events $B_k:=\{S_{N_k}(\{x_n+\xi_n\})\le
0\}$ for $k\ge 1$. As before, we apply Chebyshev's inequality to
deduce
\begin{equation*}
\P(B_k)\le \frac{\var(S_{N_k}(\{x_n+\xi_n\}))}{\left(\e
S_{N_k}(\{x_n+\xi_n\})\right)^2} \le \frac{64\sum_{n=1}^{N_k}
\sigma_n^2(\sigma_n^2+|x_n|^2)}{\left(\sum_{n=1}^{N_k}
\sigma_n^2(\sigma_n^2+|x_n|^2)\right)^2}\le
\frac{64}{\sum_{n=1}^{N_k} \sigma_n^4},
\end{equation*}
Applying \eqref{N_k_choice} we conclude that $\sum_k \P(B_k)<\infty$
and hence only finitely many of the $B_k$ occur almost surely. Thus
$T(\{x_n+\xi_n\})=1$ almost surely, as required.
\end{proof}
We remark that the detection algorithm used in the proof is based on
the log-likelihood test to distinguish between the null hypothesis
of the noisy signal distributed as pure noise and the alternative
hypothesis of the noisy signal distributed as pure noise plus an
independent sequence $\{X_n\}$, where the $\{X_n\}$ are independent
and each $X_n$ has the distribution of $\sigma_n$ times a standard
Gaussian random variable.

\subsection{}\label{sec:Walsh_and_trigonometric_non_detection} We now turn to analyze the Walsh and trigonometric
spaces. Here we prove a rather general non-detection result which
applies in particular to these spaces. Let $\X$ be the space of
signals defined by \eqref{eq:x_n_general}. We assume that $\Omega$
is endowed with a measurable structure and a probability measure
$\P$, the $\{\varphi_n\}$, $n\ge 1$, are orthogonal with respect to
$\P$ and $|\varphi_n(\omega)|\le 1$.

For the Walsh space, $\P$ will be the uniform measure on
$\{-1,1\}^\infty$ while for the trigonometric space it will be the
Lebesgue measure on $\T$. The functions on these spaces are
reordered so that the $\{\varphi_n\}$ are indexed by the natural
numbers.

\begin{proposition}\label{prop:general_non_detect_phi}
  Under the above assumptions, if the absolute values $\{|\sigma_n|\}$ are
  non-increasing and
  \begin{equation}\label{eq:sum_non_detect_cond}
    \limsup_{k\to\infty}\,\frac1{\log k} \sum_{n=1}^k |\sigma_n|^2 <
    1,
  \end{equation}
  then detection from the space $\X$ is impossible.
\end{proposition}

\begin{proof}
The result is obtained by applying the general non-detection
condition of Theorem~\ref{thm:general_detection}. Let $x(\omega),
x(\omega')$ be two random signals in $\X$, obtained by sampling
$\omega$ and $\omega'$ independently from $\P$. For each $k\ge 1$,
define the random variable
\begin{equation*}
  W_k = \sum_{n=1}^k x_n(\omega)\,\overline{x_n(\omega')} = \sum_{n=1}^k |\sigma_n|^2\,
  \varphi_n(\omega)\,\overline{\varphi_n(\omega')}.
\end{equation*}
Theorem~\ref{thm:general_detection} implies that detection from $\X$
is impossible if
\begin{equation}\label{eq:exp_W_k_cond}
  \sup_k\, \e |\exp(W_k)| < \infty.
\end{equation}
To show this, we will estimate the tail probability $\P(|W_k|>r)$.
Choose a number $a$, greater than the limsup in
\eqref{eq:sum_non_detect_cond} but smaller than $1$, and fix $m_0$
large such that
\begin{equation}\label{eq:m_0_property}
  \sum_{n=1}^m |\sigma_n|^2 \le a\log m\quad (m\ge m_0).
\end{equation}
We denote $m(r):=\exp(r/a)$, and let $r_0$ be such that $m(r_0) =
m_0$.

Fix $r\ge r_0$ and break up the sum defining $W_k$ as follows,
\begin{equation*}
  W_k = \sum_{1\le n\le m(r)} |\sigma_n|^2
  \varphi_n(\omega)\overline{\varphi_n(\omega')} + \sum_{m(r) < n\le k} |\sigma_n|^2
  \varphi_n(\omega)\overline{\varphi_n(\omega')} = W_k' + W_k'',
\end{equation*}
where it is understood that $W_k''$ is zero if $m(r)\ge k$. Using
the assumption that $|\varphi_n|\le 1$ we have
\begin{equation*}
  |W_k'| \le \sum_{1\le n\le m(r)} |\sigma_n|^2 \le a\log m(r) = r.
\end{equation*}
Thus,
\begin{equation}\label{eq:W_k_estimate}
  \P(|W_k| > r+1) \le \P(|W_k'| > r) + \P(|W_k''| > 1) = \P(|W_k''| > 1).
\end{equation}
We estimate the right-hand side using Markov's inequality and obtain
\begin{equation}\label{eq:W_k_1_estimate}
  \P(|W_k''| > 1) \le \e(|W_k''|^2) \le \sum_{m(r) < n\le k} |\sigma_n|^4,
\end{equation}
by the orthogonality of the $\{\varphi_n\}$. Since $\{|\sigma_n|\}$
is non-increasing, we may apply \eqref{eq:m_0_property} to deduce
that $|\sigma_n|^2\le a\log n / n$ for $n>m_0$. Thus,
\begin{equation}\label{eq:sigma_n_4_estimate}
  \sum_{m(r) < n\le k} |\sigma_n|^4 \le a^2\sum_{n>m(r)}
  \frac{\log^2 n}{n^2} \le \frac{Ca^2\log^2 m(r)}{m(r)} =
  Cr^2e^{-r/a},
\end{equation}
for some absolute constant $C>0$. Combining \eqref{eq:W_k_estimate},
\eqref{eq:W_k_1_estimate} and \eqref{eq:sigma_n_4_estimate} we
conclude that
\begin{equation*}
  \P(|W_k| > r+1) \le Cr^2e^{-r/a}\quad(r\ge r_0).
\end{equation*}
Finally, since $a<1$ and $r_0$ does not depend on $k$, this estimate
implies condition \eqref{eq:exp_W_k_cond}, and the proposition
follows.
\end{proof}
We remark that the above method of bounding the tail probability of
$W_k$ using an ``$\ell^1/\ell^2$'' splitting of the sum is adopted
from the paper \cite{montgomery-smith} of Montgomery-Smith.

We observe also that the non-increasing condition appearing in the
proposition cannot be removed completely. To see this, consider the
case of the Walsh space and the choice $\sigma_n=0$ if $n$ is not a
power of $2$ and $\sigma_n = \delta$ otherwise, where $\delta$ is a
small positive constant. This choice of $\{\sigma_n\}$ satisfies the
condition \eqref{eq:sum_non_detect_cond}. However, the space defined
by this choice is a Rademacher space, and hence admits detection by
Theorem~\ref{thm:detection_phases}.

\subsection{}\label{sec:Walsh_and_trigonometric_recovery}
Finally, we prove our positive result on recovery in the Walsh and
the trigonometric spaces. In both cases we rely on our general tests
for recovery given in
Section~\ref{sec:general_detection_recovery_algorithms}.

We require also the univariate case of the following lemma, a result
due to Salem and Zygmund. The multivariate case will be useful later
in the paper.
\begin{lemma}[{see \cite{kahane-random}, p.\ 70}]
\label{lemma:multivariate-salem-zygmund} Let
\begin{equation}
\label{eq:multivariate-trigonometric-random} P(t_1, \dots, t_d) =
\sum_{|n_1| + \, \cdots \, + |n_d| \leq K} c(n_1,\dots,n_d) \,
\xi(n_1,\dots,n_d) \, e^{2 \pi i (n_1 t_1 + \, \cdots \, + n_d t_d)}
\end{equation}
be a trigonometric polynomial in $d$ variables, where the $\xi(n_1,
\dots, n_d)$ are independent standard Gaussian random variables.
There is an absolute constant $C$ such that
\[
\p \Big( \|P\|_\infty \geq C \sqrt{d \log K \sum
|c(n_1,\dots,n_d)|^2} \Big) \leq K^{-2} \, e^{-d}.
\]
\end{lemma}

We start with the Walsh space. Assume that
\begin{equation}\label{eq:Walsh_recovery_assumption}
  \limsup_{k\to\infty} \frac1{\log k} \sum_{n=0}^{k-1} |\sigma_n|^2
  > 9.
\end{equation}
Let $x(\eps)$ and $x(\eps')$ be two distinct elements in the Walsh
space. We have
\begin{equation}\label{eq:Walsh_difference}
  \sum_{n=0}^{k-1} |x_n(\eps) - x_n(\eps')|^2 = \sum_{n=0}^{k-1} |\sigma_n|^2\, \big(1 -
  \prod_{j\,:\,n_j\neq 0}
  \eps_j \eps'_j\big)^2.
\end{equation}
Choose $i$ such that $\eps_i\neq \eps_i'$. Observe that if $n$ and
$m$ differ only in the $i$'th bit in the
expansion~\eqref{eq:n-representation} then exactly one of
$(1-\prod_{n_j\neq 0} \eps_j \eps'_j)$ and $(1-\prod_{m_j\neq 0}
\eps_j \eps'_j)$ is $0$ and the other is $2$. Since $|n-m|= 2^i$ and
the $\{|\sigma_n|\}$ are non-increasing, it follows that
\begin{equation*}
  |x_n(\eps) - x_n(\eps')|^2 + |x_{m}(\eps) - x_{m}(\eps')|^2 \ge 2\big(|\sigma_{n+2^i}|^2 + |\sigma_{m+2^i}|^2\big).
\end{equation*}
Hence, by pairing the summands in \eqref{eq:Walsh_difference}
according to this relation we obtain
\begin{equation}\label{eq:Walsh_difference_C_i}
  \sum_{n=0}^{k-1} |x_n(\eps) - x_n(\eps')|^2 \ge 2\sum_{n=2^i}^{k-1}
  |\sigma_n|^2 = 2\sum_{n=0}^{k-1} |\sigma_n|^2 - C_i.
\end{equation}
Observe that $C_i$ depends only on $i$ (and the sequence
$\{\sigma_n\}$).

We now check the conditions of
Corollary~\ref{cor:vol_growth_detect}. Choose an increasing sequence
$\{k_j\}$ on which the limsup in
\eqref{eq:Walsh_recovery_assumption} is attained. For each $x(\eps)$
and each $i$, define
$\X_i^{x(\eps)}:=\{x(\eps')\,:\,\eps_i'\neq\eps_i\}$. We have
\begin{equation}\label{eq:proj_k_Walsh_recovery}
  |\proj_k(\X_i^{x(\eps)})|\le |\proj_k(\X)|\le 2k.
\end{equation}
In addition, by \eqref{eq:Walsh_recovery_assumption} and
\eqref{eq:Walsh_difference_C_i},
\begin{equation}\label{eq:inf_space_Walsh_recovery}
  \inf_{x(\eps')\in\X_i^{x(\eps)}}
  \sum_{n=0}^{k_j-1} |x_n(\eps) - x_n(\eps')|^2 \ge 2\sum_{n=0}^{k_j-1} |\sigma_n|^2 -
  C_i \ge 17\log(k_j)
\end{equation}
for all but finitely many $j$. Combining
\eqref{eq:proj_k_Walsh_recovery} with
\eqref{eq:inf_space_Walsh_recovery} yields
condition~\eqref{eq:discrete_recover_entropy_cond} of
Corollary~\ref{cor:vol_growth_detect} and thus establishes recovery
in the Walsh space.

Now let $\X$ be the trigonometric space. We aim to use
Theorem~\ref{thm:recover_likelihood2} to show that $\X$ admits
recovery.

By evaluating a geometric series, for any real $\alpha$ we have
\begin{equation}\label{eq:circle_sum}
  \sum_{m\le n<m+L} |1-e^{2\pi i n\alpha}|^2 = 2L -
  2\Re\left(e^{2\pi i m \alpha} \frac{1-e^{2\pi i L\alpha}}{1-e^{2\pi i
  \alpha}}\right)
  \ge 2L - \frac{4}{|1-e^{2\pi i
  \alpha}|}.
\end{equation}
Let $x(t), x(s)$ be two distinct elements in the trigonometric
space. We have
\begin{equation}\label{eq:trig_difference}
  \sum_{|n|\le k} |x_n(t) - x_n(s)|^2 = \sum_{|n|\le k} |\sigma_n|^2\, |1-e^{2\pi i n(t-s)}|^2.
\end{equation}
Write $d(t,s):=|1-e^{2\pi i (t-s)}|$, the Euclidean distance on the
circle $\T$. By partitioning the last sum into blocks of length $L$,
using \eqref{eq:circle_sum} and the fact that $\{|\sigma_n|\}$ is
non-increasing we obtain
\begin{align}\label{eq:trig_recover_lower_bound}
  \sum_{|n|\le k} &|x_n(t) - x_n(s)|^2 \ge \sum_{j=1}^{\lfloor
  k/L\rfloor} \sum_{(j-1)L\le |n|<jL} |\sigma_n|^2\, |1-e^{2\pi i n(t-s)}|^2 \nonumber\\
  &\ge \left(2L - \frac{4}{d(t,s)}\right) \sum_{j=1}^{\lfloor
  k/L\rfloor} \left(|\sigma_{jL-1}|^2 + |\sigma_{-jL+1}|^2\right)\nonumber\\
  & \ge
  \left(2 - \frac{4}{L\cdot d(t,s)}\right) \sum_{j=1}^{\lfloor k/L
  \rfloor} \sum_{jL \le |n|<(j+1)L} |\sigma_n|^2\nonumber\\
  &  \ge \left(2 - \frac{4}{L\cdot d(t,s)}\right)\sum_{|n|\le k} |\sigma_n^2| \; - \;
  C_L,
\end{align}
where $C_L$ depends only on $L$ (and the sequence $\{\sigma_n\}$).

Let $\{k_j\}$ be an increasing sequence which will be chosen later.
Define $E_\delta(t):=\{s\in\T\,:\, d(s,t)\ge \delta\}$ for $t\in\T$.
Observe that the distance $d$ on $\T$ may be pushed forward to a
distance on $\X$ by the mapping $t\mapsto x(t)$, and the topology on
$\X$ thus obtained is the same as the topology induced from the
embedding of $\X$ in $\R^\N$. Hence, by
Theorem~\ref{thm:recover_likelihood2}, it suffices (by continuity
with respect to $s$) to show that for every $t\in\T$, the following
holds with probability one: for every $s\in E_\delta(t)$,
\begin{equation}\label{eq:trig_recover_thm_cond}
\begin{split}
  &\bigg|\sum_{|n|\le k_j} \sigma_n(e^{2\pi i ns} - e^{2\pi i n t})\xi_n\bigg|<
  \frac{1}{2}\sum_{|n|\le k_j} |\sigma_n|^2\,|e^{2\pi i ns} - e^{2\pi i n
  t}|^2,\\
  &\text{for all but finitely many $j$},
\end{split}
\end{equation}
To see this fix $t\in\T$. Then, on the one hand, by
\eqref{eq:trig_recover_lower_bound} we have
\begin{equation*}
  \inf_{s\in E_\delta}\, \frac{1}{2}\sum_{|n|\le k_j} |\sigma_n|^2\,|e^{2\pi i ns} - e^{2\pi i n
  t}|^2 \ge \frac{1}{2}\sum_{|n|\le k_j} |\sigma_n^2| - C(\delta),
\end{equation*}
provided that we take $L=L(\delta)$ sufficiently large. On the other
hand,
\begin{equation*}
  \sup_{s\in\T}\, \bigg|\sum_{|n|\le k_j} \sigma_n(e^{2\pi i ns} - e^{2\pi i n
  t})\xi_n\bigg| \le 2\sup_{s\in\T} \bigg|\sum_{|n|\le k_j} \sigma_n \xi_n e^{2\pi i
  ns}\bigg|.
\end{equation*}
By the Salem-Zygmund Lemma~\ref{lemma:multivariate-salem-zygmund}
(with $d=1$) and the Borel-Cantelli lemma, this supremum is almost
surely bounded by
\begin{equation*}
  C\bigg\{\log(2k_j+1)\sum_{|n|\le k_j}|\sigma_n|^2\bigg\}^{1/2}
\end{equation*}
for all but finitely many $j$, where $C$ is an absolute constant.
Thus, condition \eqref{eq:trig_recover_thm_cond} is satisfied with
probability one for all $s\in E_\delta$, as long as
\begin{equation}\label{eq:k_j_prop}
  \liminf_{j\to\infty}\, \frac{1}{\log k_j} \sum_{|n|\le k_j} |\sigma_n|^2
  > 4C^2.
\end{equation}
We assume that the limit in \eqref{eq:cond_non_detect_Walsh} is
greater than $4C^2$. This allows us to choose a sequence $\{k_j\}$
so that \eqref{eq:k_j_prop} is satisfied. Finally, by taking a
sequence of $\delta$'s tending to zero, and applying the above
argument to each element in the sequence, we conclude that condition
\eqref{eq:trig_recover_thm_cond} holds for any $s\in\T$, $s\neq t$.
This proves that the trigonometric space admits recovery.

We end this section with the following remark concerning the role of
the non-increasing condition in our positive results. It is not
difficult to see that this condition cannot be removed completely
from the results on recovery. E.g., for the Walsh space, if
$\sigma_n=0$ for all $n$ which are not a power of 2 (and
$\{\sigma_n\}$ is not identically zero), then the space contains two
distinct signals whose difference is finitely supported. However, it
turns out that for these spaces to admit \emph{detection}, the
condition that the limit in \eqref{eq:cond_non_detect_Walsh} is
sufficiently large suffices even without the requirement that the
$\{\sigma_n\}$ be non-increasing. This follows very similarly to the
above arguments using the results of
Section~\ref{sec:general_detection_recovery_algorithms}. We omit the
details.

\section{Detecting a trail on a tree}\label{sec:tree_trail}
\subsection{}
Let $G=(V,E)$ be the graph of an infinite rooted binary tree. We
define a space of signals $\X$ whose index set is the set $E$. In
other words, a signal in $\X$ is an assignment of numbers to the
edges of the infinite tree. Let $\E$ be the set of branches of the
tree, i.e., $\E$ is the set of infinite simple paths which start at
the root of the tree, where we consider each such path as a
collection of edges. For each path $p\in\E$ we let
$x(p)\in\{0,1\}^E$ be defined by
\begin{equation*}
  x_e(p) = \1_{(e\in p)},\quad e\in E,
\end{equation*}
and let $\X$ be the space of all such $x(p)$. In this section we
investigate the detection and recovery problems in the spaces
$\delta\X$. The following theorem establishes the existence of a
non-trivial threshold for these problems.
\begin{theorem}\label{thm:address_function}\quad
\begin{enumerate-math}
  \item If $\delta\ge \sqrt{2\log 2}$, the space $\delta\X$ admits
  recovery.
  \item If $\delta<\sqrt{2\log 2}$, the space $\delta\X$ does not
  admit detection.
\end{enumerate-math}
\end{theorem}

The detection part of this theorem appeared before in the work of
\cite{ACCHZ}, albeit with different terminology, along with results
for other graphs $G$.

We can view the above result in terms of having a critical
``signal-to-noise'' ratio for detection and recovery from the space
$\X$. Indeed, as mentioned in Section~\ref{section:preliminaries},
if $\delta x\in\delta\X$ is the transmitted signal and $z = \delta
x+\xi$ is its observed noisy version, then $\frac{1}{\delta} z = x +
\frac{1}{\delta}\xi$ is a noisy version of $x$, in which the noise
in each coordinate has variance $\frac{1}{\delta^2}$. Thus,
investigating detection and recovery in the spaces $\delta\X$ is
equivalent to investigating detection and recovery in the space $\X$
subject to different noise levels. The above result shows that for
the space $\X$ there exists a ``critical'' variance for the noise,
below which recovery is possible and above which even detection is
impossible. Other examples of such phenomena may be obtained from
the results of Section~\ref{sec:unknown_phase}.

One may use the space $\X$ to obtain examples of spaces with
different critical ``signal-to-noise'' ratios for the detection and
recovery problems. This follows from
Proposition~\ref{prop:product_spaces} by considering spaces of the
form $\delta_1\X\times\delta_2\X$.

The spaces $\delta\X$ with $\delta<\sqrt{2\log 2}$ also provide an
example of spaces in which the difference between any two distinct
signals is not in $\ell^p$ for any $p<\infty$, and yet recovery is
impossible.

One can also use the spaces $\delta\X$ to obtain a space of complex
signals which admits recovery (under complex noise), in which
neither the real nor imaginary parts admit detection. To see this,
one may consider $\X'\subseteq\C^V$ defined by $\X':=\{y\,:\,
\exists x\in\delta\X, \Re(y)=\Im(y)=x\}$ for some fixed $\sqrt{\log
2}\le \delta<\sqrt{2\log 2}$. Theorem~\ref{thm:address_function}
shows that detection is impossible from either the real or imaginary
parts of $\X'$. However, recovery from $\X'$ is possible since if
$x$ is the transmitted signal and $z$ is its noisy version, we may
obtain a noisy version of $\sqrt{2}x$ by taking
$\frac{\Re(z)+\Im(z)}{\sqrt 2}$. From this noisy version we may
recover $x$, again by Theorem~\ref{thm:address_function}.

\subsection{}
Theorem~\ref{thm:address_function} is an immediate consequence of
the next three lemmas.
\begin{lemma}\label{lem:detect_recover_equiv}
  For each $\delta>0$, the space $\delta\X$ admits detection if and
  only if it admits recovery.
\end{lemma}
\begin{proof}
  Fix $\delta>0$. Since $\delta\X\cap\ell^2=\emptyset$, it follows from Proposition~\ref{prop:recovery_implies_detection} that if $\delta\X$ admits recovery then it also admits
  detection.
  Now suppose that $\delta\X$ admits detection. Let
  $x(p)\in\delta\X$ and let $z$ be its noisy version. Denote
  by $z_\ell$ the restriction of $z$ to the left sub-tree of $G$ and by $z_r$ the
  restriction to the right sub-tree. Assume without loss of generality that the
  first step of the path $p$ is to the left sub-tree of $G$. Since the left sub-tree
  is isomorphic to the entire tree, it follows that $z_\ell$ is
distributed as a noisy version of the signal
  $x(p')$, with $p'$ the path obtained by following $p$ from its second
  step. Similarly, it follows that $z_r$ is distributed as pure noise on the whole of $E$.
  Thus, applying the detection algorithm for $\delta\X$ to both
  $z_\ell$ and $z_r$ will recover the fact that the first step of $p$
  was to the left sub-tree. Proceeding in the same manner
  iteratively allows to recover all steps of $p$.
\end{proof}

We proceed to establish the impossibility of detection from
$\delta\X$ for $\delta<\sqrt{2\log 2}$. It turns out that applying
Theorem~\ref{thm:general_detection} does not suffice for this
purpose, see Lemma~\ref{lemma:tree_detection_not_sharp} below. Thus
we employ a more precise analysis.

We make use of the following definition: For $h\ge 0$, let
$G_h=(V_h,E_h)$ be the induced subgraph of $G$ on vertices at
distance at most $h$ from the root (so that $G_h$ is a binary tree
with $2^h$ leaves).
\begin{lemma}\label{lem:tree_no_detection}
  The space $\delta\X$ does not admit
  detection if $\delta<\sqrt{2\log 2}$.
\end{lemma}
\begin{proof}
Let $P$ be the ``uniform'' probability measure on paths $p\in\E$,
i.e., the measure induced by choosing each step left or right
uniformly and independently. An inspection of the proof of
Theorem~\ref{thm:general_detection} (see the remark at the end of
Section~\ref{sec:proof_of_general_non_detection_thm}) reveals that
to show non-detection, it suffices to show that the martingale
\begin{equation*}
  f_h = \e_p \, \exp \Big\{ \sum_{e\in E_h} \Big(-\frac{x_e(p)^2}{2} +
x_e(p) \, \xi(e)\Big) \Big\} = \exp\left(-\frac{\delta^2
h}{2}\right) \e_p \, \exp ( \delta S_{p,h})
\end{equation*}
is uniformly integrable, where $\xi$ is the noise sequence,
$S_{p,h}$ is the sum of the $\xi$ variables along the path $p$ from
the root to level $h$ and $\e_p$ denotes expectation with respect to
$p$. This martingale was investigated by Biggins \cite{biggins} (see
also \cite{lyons}) in the more general context of branching random
walk (where the tree itself is random, forming a Galton-Watson
process, and the variables on the edges have a general
distribution). In \cite[Lemma 5]{biggins}, Biggins gives a necessary
and sufficient condition for the uniform integrability which we now
describe in our setting. Let $Z$ be a standard normal random
variable and define
\begin{align*}
  &m_1(\delta) = 2\e\big(\exp(-\delta Z)\1_{(Z>0)}\big),\\
  &m_2(\delta) = 2\e\big(\exp(-\delta Z)\1_{(Z<0)}\big),\\
  &m(\delta) = m_1(\delta) + m_2(\delta) = 2\exp(\delta^2/2).
\end{align*}
Since $m_1(\delta)$ and $m_2(\delta)$ are smooth functions of
$\delta$, the criterion in \cite{biggins} says that $(f_h)$ is
uniformly integrable if and only if $\e |f_1\log f_1|<\infty$ and
$m(\delta)\exp(-\delta m'(\delta)/m(\delta)) > 1$. The first
condition is easily seen to hold for all $\delta$ (e.g., since $\e
f_1^2<\infty$ for all $\delta$), whereas the second holds exactly
when $\delta<\sqrt{2\log 2}$. Thus the lemma follows.
\end{proof}

We proceed to establish the recovery criterion for the spaces
$\delta\X$. It is possible to use the general tests of
Section~\ref{sec:general_detection_recovery_algorithms} for these
spaces, however, such tests do not establish the recovery all the
way to the threshold $\sqrt{2\log 2}$. Hence we use an ad-hoc
recovery algorithm.
\begin{lemma}\label{lem:address_admits_recovery}
The space $\delta\X$ admits recovery if $\delta\ge\sqrt{2\log 2}$.
\end{lemma}
\begin{proof}
By Lemma~\ref{lem:detect_recover_equiv} it suffices to show that
$\delta\X$ admits detection when $\delta\ge\sqrt{2\log 2}$. Fix such
a $\delta$ and define the detection map $T$ as follows: For each
$p\in\E$, let $S_{p,h}(z)$ be the sum of the values of $z$ along the
path $p$ from the root to level $h$. Set $T(z)=1$ if
$\max_{p\in\E} S_{p,h^3}(z)\ge \delta h^3$ for infinitely many $h$,
and otherwise set $T(z)=0$. It is not difficult to check that $T$ is
Borel measurable. We need to show that if $z$ is a noisy version of
a signal in $\delta\X$ then $T(z)=1$ with probability one, and that
if $z$ is pure noise then $T(z)=0$ with probability one.

We start with the pure noise case. In this case, for every $p\in\E$,
$S_{p,h}(z)$ is distributed as a centered normal random variable
with variance $h$. Thus
\begin{equation*}
  \P(S_{p,h}(z)\ge \delta h)\le \frac{C}{\delta\sqrt{h}}\exp\left(-\frac{\delta^2
  h}{2}\right),
\end{equation*}
for some absolute constant $C>0$, and hence by a union bound,
\begin{equation*}
  \P\left(\max_{p\in\E} S_{p,h}(z)\ge \delta h\right) \le \frac{C2^h}{\delta
  \sqrt{h}}\exp\left(-\frac{\delta^2
  h}{2}\right) \le \frac{C}{\delta\sqrt{h}},
\end{equation*}
since $\delta\ge \sqrt{2\log 2}$. Thus the Borel-Cantelli lemma
implies that $\max_{p\in\E} S_{p,h^3}< \delta h^3$ for all but
finitely many $h$ with probability one, establishing that $T(z)=0$
almost surely.

Next, we consider the case that $z$ is a noisy version of a signal
$\delta x(p)\in\delta\X$. In this case, the process $S_{p,h}(z) -
\delta h$, indexed by $h$, forms a random walk with increments
distributed as standard normal random variables. Thus the
Hewitt-Savage 0-1 law implies that $\limsup_{h\to\infty}
S_{p,h^3}(z) - \delta h^3$ is almost surely constant. Since the
increments are symmetric and non-degenerate, we must have that this
constant is $\infty$. Thus $T(z)=1$ almost surely.
\end{proof}

\subsection{}\label{sec:gap_between_ell_2_and_uniform_integrability}
We now show that the sufficient condition for non-detection given by
Theorem~\ref{thm:general_detection} is not sharp in general, by
proving that it does not give the sharp threshold $\sqrt{2\log 2}$
for detection in the spaces $\delta\X$. Indeed, the most it can
yield, by using the ``uniform'' measure on $\X$, is that $\delta\X$
does not admit detection when $\delta<\sqrt{\log 2}$.
\begin{lemma}\label{lemma:tree_detection_not_sharp}
  For any probability measure $P$ on $\X$ and any $\delta\ge\sqrt{\log 2}$ we have
  \begin{equation*}
    \lim_{h \to\infty} \; \e \, \exp \bigg\{\delta^2\sum_{e\in E_h} x_e(p) x_e(q) \bigg\}
    = \infty,
\end{equation*}
when $x(p)$ and $x(q)$ are sampled independently from $P$.
\end{lemma}
\begin{proof}
  For paths $p,q\in\E$ define $N(p,q)$ to be the number of
  edges common to $p$ and $q$. Let $\partial V_h$ be the $2^h$ edges of $E_h$ incident to the leaves of the truncated tree $G_h$.
  Observe that if $x(p)$ and $x(q)$ are sampled independently
  from any probability measure $P$ on $\X$ then
  \begin{equation*}
    \P(N(p,q) \ge h) = \sum_{e\in \partial V_h} \P(e\in p\cap q) =
    \sum_{e\in \partial V_h} \P(e\in p)^2 \ge 2^{-h},
  \end{equation*}
  by the fact that $\sum_{e\in \partial V_h} \P(e\in
  p) = 1$ and the Cauchy-Schwartz inequality.
 Thus,
  \begin{align*}
    \e \, &\exp \bigg\{\delta^2\sum_{e\in E_h} x_e(p) x_e(q) \bigg\} = \e \, \exp \bigg\{\delta^2
    \min(N(p,q),h)
    \bigg\}\\
    &= 1 + \sum_{i=1}^{h} \P(N(p,q)\ge i) \big(\exp(\delta^2 i) -
    \exp(\delta^2 (i-1))\big)
    \ge \sum_{i=0}^{h} 2^{-(i+1)}\exp(\delta^2 i),
  \end{align*}
  which tends to infinity as $h\to\infty$, when $\delta \ge \sqrt{\log 2}$.
\end{proof}


\section{Almost periodic and polynomial phase functions}
\label{sec:poly_phase}

\subsection{}
The space of almost periodic functions on $\Z$ may be defined as the
uniform closure (i.e., closure in the $\ell^\infty$ norm) of the
linear combinations of functions of the form $e^{2 \pi i n \theta}$,
$\theta \in \T$. This is a translation-invariant algebra of
functions which we denote by $\AP$. In this section we show that the
space of almost periodic functions admits recovery, thus proving
Theorem~\ref{thm:reconstruction-polynomial-intro}.

In addition, we briefly consider the space of polynomial phase
functions. Following \cite{furstenberg:recurrence} we denote this
space by $\W$ (after Weyl). It is defined as the uniform closure of
the linear combinations of functions on $\Z$ of the form $e^{2 \pi i
p(n)}$, where $p(x)$ is a real polynomial. This is a
translation-invariant algebra of functions, which contains the
almost periodic functions. We leave open the question of whether
$\W$ admits recovery but give some motivation for why this may be
the case.

We shall consider complex-valued noise since our spaces contain
complex-valued signals.

\subsection{Recovery of almost periodic functions. Proof of Theorem~\ref{thm:reconstruction-polynomial-intro}}
In this section we describe a Borel measurable recovery mapping for
the space $\AP$ of almost periodic functions.

The ``mean value'' of a function $f$ may be defined as the limit
\begin{equation}\label{eq:mean_value_def}
M\{f(n)\} := \lim_{N \to \infty} \frac1{2N+1} \sum_{-N}^{N} f(n).
\end{equation}
This limit exists for any $f \in \AP$. This is easy to see directly
when $f$ has the form $e^{2 \pi i n\theta}$, $\theta\in\T$ (in fact,
the limit equals $0$ unless $\theta=0$). It then extends simply to
linear combinations of such functions and finally to the whole space
$\AP$, as the uniform closure of these linear combinations.

Since $\AP$ is an algebra, the latter implies that
\begin{equation}\label{eq:additional_mean_values_def}
M\{f(n) \overline{g(n)} \} \quad \text{and} \quad M\{|f(n)|^2\}
\end{equation}
also exist, for any $f, g \in \AP$. In
Section~\ref{lemma:mean_value_recovery_polynomial_phase} below we
show that the above limits exist also in the larger space $\W$.

Define the auto-correlation mapping
\begin{equation*}
  \AC_k(g):=M\{g(n)\overline{g(n-k)}\},\quad g\in\C^\Z,
\end{equation*}
where we set $\AC_k(g)=0$ if the limit in the definition of $M$ does
not exist.

Define inductively the mapping, for $j\ge 1$ and $g\in\C^\Z$,
\begin{equation}\label{eq:L_g_j_def}
\begin{split}
  L_0(g)&:=0,\\
  L_j(g)&:=\min\{k>L_{j-1}(g)\,:\,|\AC_k(g) - \AC_0(g) - 2|\le
  1/j\},
\end{split}
\end{equation}
where we set $L_j(g)=0$ if there is no $k$ satisfying the condition.
By the translation invariance of $\AP$ it suffices to present a
recovery mapping for the central element of the transmitted signal.
We define this recovery mapping by
\begin{equation}\label{eq:almost_periodic_recovery}
  T(g):=\lim_{m\to\infty} \frac{1}{m}\sum_{j=1}^m g(L_j(g)),\quad
  g\in\C^\Z,
\end{equation}
where again, we set $T(g)=0$ if the above limit does not exist. The
mapping $T$, as well as $A_k$ and $L_j$, are easily seen to be Borel
measurable. It remains to prove that for every $f\in \AP$ we have
$T(f + \xi) = f(0)$ almost surely.

We make use of the well-known fact that a function $f$ is almost
periodic if and only if
for every $\eps>0$ there exists a syndetic set $(M_j)\subset\Z$ of
\emph{$\eps$-almost periods} (see, e.g., \cite[Chapter
5]{katznelson}). Here, a syndetic set is a set with bounded gaps and
an integer $M$ is called an $\eps$-almost period if
\begin{equation*}
  \sup_n |f(n+M) - f(n)|\le \eps.
\end{equation*}

The following sequence of lemmas concludes the proof.
\begin{lemma}\label{lem:auto_correlation_with_noise}
  For each $f\in\AP$ and each integer $k$, almost surely,
  \begin{equation*}
    \AC_k(f+\xi) = \begin{cases} \AC_k(f)&k\neq
    0\\\AC_0(f)+2&k=0\end{cases}.
  \end{equation*}
\end{lemma}
\begin{proof}
  Fix $f\in\AP$ and $k\in\Z$. Let $g:=f+\xi$. Observe that
  \begin{equation*}
    g(\cdot)\overline{g(\cdot-k)} = f(\cdot)\overline{f(\cdot-k)} +
    f(\cdot)\overline{\xi(\cdot-k)} + \xi(\cdot)\overline{f(\cdot-k)} +
    \xi(\cdot)\overline{\xi(\cdot-k)} =: g_1 + g_2 + g_3 + g_4.
  \end{equation*}
  Thus it suffices to show that,
  almost surely,
  \begin{equation*}
    M\{g_i\}=0,\quad i=2,3
  \end{equation*}
  and
  \begin{equation}\label{eq:M_g_4}
    M\{g_4\} = \begin{cases}
      0&k\neq 0\\
      2&k=0
    \end{cases}.
  \end{equation}
  We start with $g_2$. Since $f\in\ell^\infty$ it follows that for any $\eps>0$ and
  all $N\ge 1$,
  \begin{equation*}
    \P\left(\frac1{2N+1} \left|\sum_{n=-N}^{N}
    f(n)\overline{\xi(n-k)}\right|>\eps\right) \le
    \frac{1}{((2N+1)\eps)^4}\e\left(\left|\sum_{n=-N}^{N}
    f(n)\overline{\xi(n-k)}\right|^4\right)\le
    \frac{C(f)}{N^{2}\eps^4}.
  \end{equation*}
  for some constant $C(f)$. Thus, the Borel-Cantelli lemma implies
  that $M\{g_2\}=0$ almost surely. A similar argument shows that
  $M\{g_3\}=0$ almost surely.

  It remains to show that \eqref{eq:M_g_4} occurs almost surely. If
  $k=0$ we have $g_4(n) = |\xi(n)|^2$ and hence $M\{g_4\} = 2$
  almost surely by the strong law of large numbers, upon recalling that
  $\e|\xi(n)|^2 = 2$ (since we have complex-valued noise). If $k\neq 0$ we may write
  \begin{equation*}
    \frac1{2N+1} \sum_{n=-N}^{N}
    \xi(n)\overline{\xi(n-k)} = \frac1{2|k|} \sum_{r=0}^{2|k|-1}
    \frac{2|k|}{2N+1}\sum_{\substack{-N\le n\le N\\ n\equiv r \bmod{2|k|}}}
    \xi(n)\overline{\xi(n-k)}.
  \end{equation*}
  Then, by considering for each fixed $r$ the limit of the expression inside the outer sum, we
  conclude that $M\{g_4\}=0$ almost surely by the strong law of large numbers.
\end{proof}
\begin{lemma}\label{lem:f_recovery_lemma}
For $f\in\AP$ and $j\ge 0$ define inductively
\begin{equation}\label{eq:L_prime_f_def}
\begin{split}
  L'_0(f)&:=0,\\
  L'_j(f)&:=\min\{k> L_{j-1}(f)\,:\,|\AC_k(f) - \AC_0(f)|\le 1/j\}.
\end{split}
\end{equation}
Then
\begin{equation*}
  \sup_n |f(n) - f(n + L'_j(f))| \to 0\quad\text{as $j\to \infty$}.
\end{equation*}
\end{lemma}
\begin{proof}
First, note that the limit in the definition of $\AC_k(f)$ exists
for all $k$ since $M\{f(n)\}$ exists for all $f\in\AP$ and the
almost periodic functions form an algebra. Next, observe that for
any $k$,
\begin{equation*}
  M\{|f(n) - f(n-k)|^2\} = 2(\AC_0(f) - \Re(\AC_k(f))).
\end{equation*}
Also, the Cauchy-Schwartz inequality implies that $|\AC_k(f)|\le
\AC_0(f)$. Thus, recalling that $f$ has a syndetic set of
$\eps$-almost periods for every $\eps$ implies that the minimum in
\eqref{eq:L_prime_f_def} is taken over a non-empty set. In addition,
we conclude that
\begin{equation}\label{eq:z_j_mean_energy_conv}
  M\{|f(n) - f(n+L'_j(f))|^2\} \le 2/j.
\end{equation}

Let us denote by $z_j$ the function $\{f(n) - f(n+L'_j(f))\}$. We
may interpret \eqref{eq:z_j_mean_energy_conv} as saying that the
$z_j$ converge to zero in \emph{mean energy}. We wish to conclude
from this that they also tend to zero in the uniform norm. This may
be deduced, for instance, using the following well-known
characterization of almost-periodic functions. A function $f$ is
almost-periodic if and only if the set of its translates
$(f(\cdot-k))$ is precompact in $\ell^\infty$ \cite[Remark 1.8, p.
139]{petersen}.
\end{proof}
\begin{lemma}
  The recovery mapping defined by
  \eqref{eq:almost_periodic_recovery} satisfies that for each
  $f\in\AP$, almost surely, $T(f+\xi) = f(0)$.
\end{lemma}
\begin{proof}
  Fix $f\in\AP$ and set $g := f+\xi$. Recalling the definitions of $L_j(g)$ and $L'_j(f)$ from \eqref{eq:L_g_j_def} and \eqref{eq:L_prime_f_def}, and putting together
  Lemmas~\ref{lem:auto_correlation_with_noise} and
  \ref{lem:f_recovery_lemma} we see that, almost surely, $L_j(g) =
  L'_j(f)$ for all $j$. In particular, the $L_j(g)$ are independent of $\xi$. Since the $L_j(g)$ are also strictly increasing, Lemma~\ref{lem:f_recovery_lemma}
  and the strong law of large numbers imply that
  \begin{equation*}
    T(g) = \lim_{m\to\infty} \frac{1}{m}\sum_{j=1}^m
    g(L(g)_{j}) = f(0) + \lim_{m\to\infty} \frac{1}{m}\sum_{j=1}^m
    \xi(L(g)_{j}) = f(0),
  \end{equation*}
  establishing the lemma.
\end{proof}

\subsection{Polynomial phase
functions}\label{sec:polynomial_phase_functions_subsec} In the rest
of the section we discuss the space $\W$ of polynomial phase
functions. We explain how a polynomial phase function $f\in\W$ is
uniquely determined by a certain set of parameters, and how these
parameters may be recovered almost surely from a noisy version of
$f$. This falls short of providing a measurable recovery mapping for
$\W$ since it is not clear how to measurably recover $f$ from its
parameter set. However, it gives some indication that the space $\W$
should admit recovery.

The following pair of lemmas summarize our results.
\begin{lemma}
\label{lemma:polynomial-uniqueness} A function $f \in \W$ is
uniquely determined by the values
\begin{equation}
\label{eq:polynomial-measurement} M\{f(n) \, e^{-2 \pi i p(n)}\},
\end{equation}
where $p(x)$ goes through all real polynomials.
\end{lemma}
\begin{lemma}\label{lemma:mean_value_recovery_polynomial_phase}
  For any $f\in\W$ we have, almost surely, that
  \begin{equation*}
    M\{(f(n)+\xi(n))e^{-2 \pi i p(n)}\} = M\{f(n) \, e^{-2 \pi i p(n)}\}
  \end{equation*}
  simultaneously for all real polynomials $p(x)$.
\end{lemma}
\subsection{Preliminary
results}\label{sec:polynomial_phase_functions_preliminaries} We rely
on the following results.

\begin{theorem}[Weyl \cite{weyl}]
\label{thm:weyl} If $p(x)$ is a real polynomial with at least one
coefficient other than the constant term irrational, then the
sequence $\{p(n)\}$ is equidistributed modulo one.
\end{theorem}

\begin{theorem}[Furstenberg \cite{furstenberg:recurrence}]
\label{thm:furstenberg-polynomial} Let $p_1(x), \dots, p_k(x)$ be
real polynomials. For any $\eps > 0$, the set of integers $n$
satisfying simultaneously
\begin{equation}
\label{eq:furstenberg-polynomial} \big| e^{2 \pi i p_j(n)} - e^{2
\pi i p_j(0)} \big| < \eps \quad (j=1,\dots,k)
\end{equation}
is syndetic (i.e. has bounded gaps).
\end{theorem}

For more details on these theorems we refer to the book
\cite{furstenberg:recurrence} (see pp.\ 31 and 69).

The ``mean value'' of a function $f$, defined by
\eqref{eq:mean_value_def}, exists for any $f \in \W$. To see this,
suppose first that $f(n)$ has the form $e^{2 \pi i p(n)}$, where
$p(x)$ is a real polynomial. If $p(x)$ has at least one coefficient
other than the constant term irrational, then Theorem \ref{thm:weyl}
implies that $M\{f(n)\}$ exists and is equal to zero. If $p(x)$ has
only rational coefficients (except for, possibly, the constant term)
then $f(n)$ is a periodic function, and it follows again that
$M\{f(n)\}$ exists. This extends easily to linear combinations of
functions of the form $e^{2 \pi i p(n)}$ and finally to the uniform
closure of these functions, namely to the whole space $\W$.

Since $\W$ is an algebra, the latter implies that the limits in
\eqref{eq:additional_mean_values_def} also exist, for any $f, g \in
\W$.

\begin{lemma}
\label{lemma:polynomial-parseval} If $f \in \W$ and is not
identically zero, then $M\{|f(n)|^2\}$ is non-zero.
\end{lemma}

\begin{proof}
By translation invariance we may suppose that $f(0) \neq 0$. Given
$\delta > 0$ we can find an element $g \in \W$ of the form
\begin{equation}
\label{eq:polynomial-approximation} g(n) = \sum_{j=1}^{k} c_j \,
e^{2 \pi i p_j(n)} \quad (n \in \Z)
\end{equation}
and such that $\|f - g\|_\infty < \delta$. Given $\eps > 0$ we use
Theorem \ref{thm:furstenberg-polynomial} to find a syndetic set $A
\subset \Z$ such that \eqref{eq:furstenberg-polynomial} is satisfied
for each $n \in A$. It follows that
\[
|g(n) - g(0)| \leq \eps \sum_{j=1}^{k} |c_j|, \quad n \in A.
\]
If we choose $\eps = \eps(\delta,g)$ small enough, this implies that
$|f(n) - f(0)| < 3\delta$ for $n \in A$. If $\delta = \delta(f)$ is
also chosen sufficiently small, it follows that $|f(n)|$ is bounded
away from zero on a syndetic set, and so $M\{|f(n)|^2\}$ cannot
vanish.
\end{proof}

\subsection{Proof of Lemma~\ref{lemma:polynomial-uniqueness}}
Suppose that the lemma is not true, then there exists a function $f
\in \W$ not identically zero, such that all the values
\eqref{eq:polynomial-measurement} vanish. Given $\delta > 0$ we can
find an element $g \in \W$ of the form
\eqref{eq:polynomial-approximation} and such that $\|f - g\|_\infty
< \delta$. We have
\[
\big| M\{|f(n)|^2\} - M\{f(n) \overline{g(n)}\} \big| = \big|
M\{f(n) \overline{(f(n)-g(n))}\} \big| \leq \delta \|f\|_\infty.
\]
According to Lemma \ref{lemma:polynomial-parseval}, $M\{|f(n)|^2\}$
is non-zero, so if we choose $\delta$ small enough this implies that
$M\{f(n) \overline{g(n)}\}$ is also non-zero. But, on the other
hand, $M\{f(n) \overline{g(n)}\}$ is a finite linear combination of
values of the form \eqref{eq:polynomial-measurement}, hence it must
vanish, a contradiction.

\subsection{Proof of Lemma~\ref{lemma:mean_value_recovery_polynomial_phase}} Fix $f \in \W$. Given a
real polynomial
\[
p(x) := a_1 x + a_2 x^2 + \cdots + a_d x^d
\]
with vanishing constant term, we consider the finite averages
\begin{equation*}
M_N(f,p):=\frac{1}{2N+1} \sum_{-N}^N f(n) \, e^{-2\pi i p(n)}.
\end{equation*}
Let $g := f + \xi$ denote a noisy version of $f$, then we have
\[
M_N(g,p) = M_N(f,p) + M_N(\xi,p).
\]
We may view the last term
\[
M_N(\xi,p) = \frac1{2N+1} \sum_{-N}^{N} \xi(n) \, e^{-2\pi i (a_1 n
+ a_2 n^2 + \cdots + a_d n^d)}
\]
as a random trigonometric polynomial in $d$ variables $a_1, \dots,
a_d$. We wish to apply Lemma~\ref{lemma:multivariate-salem-zygmund}
to it. This trigonometric polynomial has the form
\eqref{eq:multivariate-trigonometric-random}, with the coefficients
$c(n_1,\dots,n_d)$ being zero unless the vector $(n_1,\dots,n_d)$ is
of the form $(n,n^2,\dots,n^d)$, $|n| \leq N$. We may therefore
apply Lemma \ref{lemma:multivariate-salem-zygmund} with, say, $K =
N^{d+1}$. The lemma shows that, with probability at least $1 -
N^{-2(d+1)} e^{-d}$, we have
\[
\sup_{p} |M_N(\xi,p)| \leq C(d) \sqrt{\frac{\log N}{N}},
\]
where the supremum is taken over all real polynomials $p(x)$ with
vanishing constant term and with degree not greater than $d$. Using
the Borel-Cantelli lemma, and since every polynomial has a finite
degree, this implies that the condition
\[
\lim_{N \to \infty} M_N(\xi,p) = 0 \quad \text{for every real
polynomial $p(x)$}
\]
is satisfied with probability one (also for polynomials with
non-vanishing constant term, certainly). Hence $M_N(g,p)$ and
$M_N(f,p)$ both converge to the same value, namely to the value
\eqref{eq:polynomial-measurement}, simultaneously for all real
polynomials $p(x)$.


\section{Alternative models}
\label{sec:alternative_models}

In this section we provide a brief discussion of some variations on
our setup. We consider the effect of having noise distributions
other than the Gaussian and touch briefly on two alternative models
for the detection and recovery concepts.

\subsection{Alternative noise distributions}
In this section we comment briefly on the situation when the noise
sequence $(\xi_n)$ is independent and identically distributed with a
distribution other than the Gaussian. We denote the common
distribution of $\xi_n$ by $\mu$. We focus on the case that the
signals and noise are real-valued.

\subsubsection{Simple necessary conditions} As mentioned in the introduction, a necessary condition
for detection under Gaussian noise is that
\begin{equation}\label{eq:detection_nec_cond_alt_noise}
\X\cap\ell^2=\emptyset,
\end{equation}
and this condition is also sufficient when $\X$ is countable. Shepp
\cite{shepp} considered a general noise distribution $\mu$ and
investigated the singularity of the measures of $\xi$ and $x+\xi$
for a fixed sequence $x$. The results of \cite{shepp} imply the
following.
\begin{enumerate-math}
  \item Condition \eqref{eq:detection_nec_cond_alt_noise} remains necessary for
detection under any noise distribution $\mu$.
  \item If $\mu$ has finite Fisher information, that is, if $\mu$ is absolutely
continuous and its density $f$ is almost everywhere positive,
locally absolutely continuous and satisfies
\begin{equation*}
  \int_{-\infty}^\infty \frac{(f')^2}{f} < \infty,
\end{equation*}
then condition \eqref{eq:detection_nec_cond_alt_noise} is sufficient
for detection when $\X$ is countable.
  \item If $\mu$ does not have finite Fisher information, then there exists some signal
$x\notin\ell^2$ such that the space $\X=\{x\}$ does not admit
detection.
\end{enumerate-math}
Similar statements follow for the recovery problem with respect to
the condition $(\X-\X)\cap\ell^2=\{0\}$.

To illustrate what may replace condition
\eqref{eq:detection_nec_cond_alt_noise} when the measure $\mu$ does
not have finite Fisher information one may consider the case that
$\mu$ is the uniform measure on $[-1,1]$. For this noise
distribution, the distributions of the signals $\xi$ and $x+\xi$,
for a fixed sequence $x$, are singular if and only if either
\begin{equation}\label{eq:detection_nec_cond_uniform_dist}
\sup |x_n| \ge 2\text{ or }\sum |x_n|=\infty.
\end{equation}
This may be verified using Kakutani's dichotomy for product measures
\cite{kakutani} (see also the version \cite[Theorem 2]{shepp}). It
follows that the space $\X=\{x\}$ admits detection, when $\mu$ is
uniform on $[-1,1]$, if and only if condition
\eqref{eq:detection_nec_cond_uniform_dist} is satisfied. That is,
the $\ell^2$ condition \eqref{eq:detection_nec_cond_alt_noise} is
replaced by an $\ell^\infty$ and an $\ell^1$ condition.

\subsubsection{Convolution of two noise distributions}\label{sec:conv_general_noise} Suppose the distribution $\mu$ can be written as a
convolution of two distributions $\mu_1$ and $\mu_2$. Then any space
$\X$ which does not admit detection under $\mu_1$ or $\mu_2$ does
not admit detection also under $\mu$. To see this, informally,
suppose $\X$ admits detection under $\mu$. Then one may use the
following detection algorithm to show that $\X$ admits detection
under $\mu_1$, say. Upon receiving a signal with noise having
distribution $\mu_1$, one may add to it an additional, independent,
noise with the distribution $\mu_2$ and then apply the detection
algorithm corresponding to $\mu$. A formal argument along these
lines may be obtained as in the proof of
Proposition~\ref{prop_monotonicity_in_noise}. A similar statement is
true for the recovery problem.

This remark may be particularly useful for noise distributions $\mu$
which are the convolution of the standard Gaussian distribution with
another distribution. For such noises, we may apply directly the
non-detection and non-recovery results of this paper. Observe that
this class of noises includes distributions with heavy tails but
excludes distributions with bounded support.

\subsubsection{A condition for non-detection} The non-detection condition given by
  Theorem~\ref{thm:general_detection} has the following
  analogue for general noise distributions.
  \begin{theorem}\label{thm:general_detection_general_noise}
    Let $\X$ be a Borel subset of $\R^\N$. Let $\mu$ be an absolutely continuous distribution with an almost everywhere positive density $f$. Suppose that there is a
probability measure $P$ on $\X$ such that
\begin{equation}\label{eq:non_detection_cond_general_noise}
  \liminf_{k \to\infty} \; \e  \left[\,\prod_{n=1}^k \int \frac{f(z - x_n)f(z - y_n)}{f(z)} dz\right] <
  \infty,
\end{equation}
where $\{x_n\}$ and $\{y_n\}$ are sampled independently from $P$.
Then $\X$ does not admit detection under the noise distribution
$\mu$.
\end{theorem}
This is proved in the same way as
Theorem~\ref{thm:general_detection}.

\subsubsection{Critical signal-to-noise ratio} An interesting phenomenon present for the Gaussian
noise distribution is that of the critical ``signal-to-noise''
ratio, as presented, for instance, in
Theorem~\ref{thm:address_function}. It turns out that this
phenomenon is present for a rather large class of noise
distributions, as the following result shows.
\begin{theorem}
Let $\X$ be the space presented in Section~\ref{sec:tree_trail}. Let
$\mu$ be an arbitrary distribution.
\begin{enumerate-math}
  \item \label{part:recovery_tree_general_noise}There exists a $\delta_0(\mu)>0$ such that for
  any $\delta\ge \delta_0(\mu)$ the space $\delta\X$ admits
  recovery under the noise distribution $\mu$.
  \item \label{part:non-detection_tree_general_noise}Suppose that $\mu$ is absolutely continuous with an almost everywhere positive density $f$.
If $\delta$ satisfies
\begin{equation}\label{eq:mu_1_delta_def}
  \int \frac{f(z - \delta)^2}{f(z)} dz < 2
\end{equation}
  then the space $\delta\X$ does not
  admit detection under the noise distribution $\mu$.
\end{enumerate-math}
\end{theorem}
We remark that for many natural distributions, such as the Cauchy or
Gaussian distribution, condition~\eqref{eq:mu_1_delta_def} is
satisfied for all $\delta$ in some open interval containing $0$.
Asking condition~\eqref{eq:mu_1_delta_def} to be satisfied in a
neighborhood of $0$ bears some resemblance to having finite Fisher
information although neither condition implies the other (in one
direction consider a density proportional to $\exp(-\exp(x^2))$ and
in the other direction consider a discontinuous density).
\begin{proof}
We start with part \ref{part:recovery_tree_general_noise}. Fix $a>0$
such that
\begin{equation}\label{eq:xi_a_p_cond}
  \P(\xi(0) \ge a)\le 1/5.
\end{equation}
Define $\delta_0(\mu)$ by the condition that for any $\delta\ge
\delta_0(\mu)$,
\begin{equation}\label{eq:delta_xi_p_cond}
  \P(\delta + \xi(0) < a) \le 1/3.
\end{equation}
It suffices to show that $\delta\X$ admits detection if $\delta\ge
\delta_0(\mu)$ since it will then follow that $\delta\X$ admits
recovery by the argument of Lemma~\ref{lem:detect_recover_equiv}.
Recall the definition of the space $\X$ from
Section~\ref{sec:tree_trail} and, in particular, the definition of
$\E$ as the set of branches of the underlying binary tree. Denote by
$z$ the noisy signal. Define the detection map $T$ as follows: For
each branch $p\in\E$ and $h\ge 1$ consider the elements of $z$ along
the branch $p$ from the root to level $h$. Set $S_{p,h}(z)$ to be
the number of these elements which are larger or equal to $a$. Set
$T(z)=1$ if there exists a $p\in\E$ such that $S_{p,h}(z)\ge h/2$
for all but finitely many values of $h$. Otherwise set $T(z)=0$.

To check the validity of this detection map suppose first that $z$
is a noisy version of the signal corresponding to the path $p\in\E$.
Then $S_{p,h}(z)$ is a sum of $h$ independent Bernoulli random
variables, each with probability at least $2/3$ to be $1$, by
\eqref{eq:delta_xi_p_cond}. Thus, the strong law of large numbers
implies that $S_{p,h}(z)\ge h/2$ for all but finitely many $h$, so
that $T(z)=1$ almost surely.

Now suppose that $z$ is pure noise. Then, for any path $p\in\E$ and
any $h\ge 1$, $S_{p,h}(z)$ is a sum of $h$ independent Bernoulli
random variables, each with probability at most $1/5$ to be $1$, by
\eqref{eq:xi_a_p_cond}. By a union bound, the chance that there
exists a path $p\in\E$ such that $S_{p,h}(z)\ge h/2$ is at most
$(4/5)^h$. Thus, by the Borel-Cantelli lemma, $T(z)=0$ almost
surely.

We continue with part \ref{part:non-detection_tree_general_noise}.
We wish to use Theorem~\ref{thm:general_detection_general_noise}.
Let $P$ be the measure induced on signals in $\delta \X$ by choosing
a path $p\in\E$ according to the ``uniform'' measure on paths, as in
the proof of Lemma~\ref{lem:tree_no_detection}. Denote by $E_h$ the
set of edges of the underlying tree up to level $h$. By
Theorem~\ref{thm:general_detection_general_noise} it suffices to
show that
\begin{equation*}
  \liminf_{h \to\infty} \; \e  \left[\,\prod_{e\in E_h} \int \frac{f(z - x_e)f(z - y_e)}{f(z)} dz\right] <
  \infty,
\end{equation*}
where $\{x_e\}$ and $\{y_e\}$ are sampled independently from $P$.
Recall that $x_e, y_e\in\{0,\delta\}$. Thus, if either $x_e$ or
$y_e$ equals zero we have
\begin{equation*}
  \int \frac{f(z - x_e)f(z - y_e)}{f(z)} dz = \int f(z)dz = 1.
\end{equation*}
It follows that
\begin{equation*}
  \prod_{e\in E_h} \int \frac{f(z - x_e)f(z - y_e)}{f(z)} dz
  = \left[\int \frac{f(z - \delta)^2}{f(z)}
  dz\right]^{\min(N(x,y),h)},
\end{equation*}
where $N(x,y)$ is the number of edges common to the paths defining
$x$ and $y$. Thus, noting that $\P(N(x,y)= k)=2^{-(k+1)}$ we deduce
that for any $\delta$ satisfying \eqref{eq:mu_1_delta_def},
\begin{equation*}
  \liminf_{h \to\infty} \; \e  \left[\,\prod_{e\in E_h} \int \frac{f(z - x_e)f(z - y_e)}{f(z)} dz\right] =
  \e \left[\int \frac{f(z - \delta)^2}{f(z)}
  dz\right]^{N(x,y)}
  <\infty.
\end{equation*}
Thus, for these $\delta$,
Theorem~\ref{thm:general_detection_general_noise} implies that
$\delta\X$ does not admit detection.
\end{proof}
The proof used the fact that detection implies recovery for the
space $\X$, by the arguments of
Lemma~\ref{lem:detect_recover_equiv}. The converse is also true.
Indeed, as in Proposition~\ref{prop:recovery_implies_detection}, one
need only exhibit a recovery mapping for each of the two-point
spaces $\X_p:=\{0,x(p)\}$, $p\in\E$, with this mapping depending
measurably on $x(p)$. This reduces to distinguishing a non-zero
constant signal from the zero signal, which is certainly possible
under any noise distribution.

Our methods do not show the existence of a threshold $\delta_c(\mu)$
such that $\delta\X$ admits recovery if $\delta>\delta_c(\mu)$ and
does not admit detection if $\delta<\delta_c(\mu)$. However, for
noise distributions which are part of a semigroup, such as the
Cauchy distribution, one may deduce the existence of such a
threshold from the remarks in Section~\ref{sec:conv_general_noise}
together with the fact that detection and recovery are equivalent
for the spaces $\delta\X$.

\subsection{Uniform recovery and detection}
According to our definitions, for $\X$ to admit recovery via the map
$T$ we require that for every $x\in\X$, $T(x + \xi) = x$ almost
surely. Thus, we allow an exceptional set of probability zero of
noises on which the recovery may fail, and this exceptional set may
depend on the signal $x$ being recovered. One may also consider a
uniform version of the recovery problem, in which the exceptional
set is required to be the same for all possible signals. In other
words, one may ask that with probability one, the recovery mapping
succeeds for all signals $x\in\X$ simultaneously. We focus on the
real-valued case.
\begin{definition*}
Let $I$ be a countable set. We say that a space $\X\subseteq\R^I$
\emph{admits uniform recovery} if there exists a Borel measurable
function $T:\R^I\to\R^I$ with the property that, almost surely, $T(x
+ \xi) = x$ for all $x\in \X$.
\end{definition*}
For this definition to make sense we need that the set of $\xi$ for
which $T(x + \xi) = x$ for all $x\in \X$ be measurable. We note that
when the space $\X$ is Borel, this set is co-analytic and hence
universally measurable.

Certainly, to admit uniform recovery is a more stringent requirement
than to admit (non-uniform) recovery. We point out, however, that
\emph{linear} recovery maps work equally well for both notions.
Here, by a linear recovery map we mean a recovery map $T$ for which
there exists a set $A$ of noises having probability one such that
$T(x + \xi) = T(x) + T(\xi)$ for all $x\in\X$, $\xi\in A$. Several
of the recovery maps introduced earlier are of this type, including
the maps used in Section~\ref{sec:hausdorff} for recovery of Fourier
transforms of measures, both with known and with unknown support,
and including the recovery map in
Lemma~\ref{lemma:mean_value_recovery_polynomial_phase} for the
parameters~\eqref{eq:polynomial-measurement} of polynomial phase
functions.

It is natural to ask whether our two notions of recovery in fact
coincide. We do not answer this question here, but show that the
answer is negative in the context of the detection problem.

\begin{definition*}
Let $I$ be a countable set. We say that a space $\X\subseteq\R^I$
\emph{admits uniform detection} if there exists a Borel measurable
function $T:\R^I\to\{0,1\}$ satisfying that, almost surely, $T(\xi)
= 0$ and $T(x + \xi) = 1$ for all $x\in \X$.
\end{definition*}

Again, we emphasize that this definition differs from our standard
one in that we require that with probability one, the detection map
succeeds for all $x\in\X$.

For $c>0$, define
\begin{equation*}
\X_c:=\Big\{x\in\R^\N\colon \lim_{N\to\infty} \frac{1}{N}
\sum_{n=1}^N x_n^2 = c^2\Big\}.
\end{equation*}
The following theorem shows that for $0<c\le 2$, detection in $\X_c$
is possible whereas uniform detection is not.

\begin{theorem}\label{thm:uniform_detection} \quad
  \begin{enumerate-math}
  \item \label{item:uniform_detection_nonuniform}
     $\bigcup_{c>0}\X_c$ admits (non-uniform) detection.
    \item \label{item:uniform_detection_possible}
    $\bigcup_{c>2}\X_c$ admits uniform detection.
     \item \label{item:uniform_detection_impossible}
    For $0<c\le 2$, $\X_c$ does not admit uniform detection.
 \end{enumerate-math}
\end{theorem}

\begin{proof}[Proof of parts \ref{item:uniform_detection_nonuniform}
 and \ref{item:uniform_detection_possible}]
Define the mapping $T$ to equal 1 on the sequence $y$ if and only if
\begin{equation*}
  \limsup_{N\to\infty} \frac{1}{N}\sum_{n=1}^N y_n^2 > 1.
\end{equation*}
Let us check that $T$ solves the (non-uniform) detection problem for
$\bigcup_{c>0} \X_c$. By the strong law of large numbers,
$\frac{1}{N}\sum_{n=1}^N \xi_n^2 \to 1$ almost surely. Thus
$T(\{\xi_n\})=0$ almost surely. Now let $x\in \X_c$ for some $c>0$.
We have
\begin{equation*}
\lim_{N\to\infty} \frac{1}{N}\sum_{n=1}^N (x_n + \xi_n)^2
 = c^2+1+\lim_{N\to\infty} \frac{2}{N}\sum_{n=1}^N x_n\xi_n
\end{equation*}
almost surely. Thus it remains only to note that
$X_N:=\frac{1}{N}\sum_{n=1}^N x_n\xi_n$ is distributed as a normal
random variable with expectation zero and variance
\begin{equation*}
  \frac{1}{N^2} \sum_{n=1}^N x_n^2 = O(N^{-1}),
\end{equation*}
since $x\in \X_c$. Thus $X_N\to 0$ almost surely. Since $c>0$ we
conclude that $T(\{x_n+\xi_n\})=1$ almost surely, as required.

To show that $T$ also solves the uniform detection problem for
$\bigcup_{c>2} \X_c$ we follow the same steps and need only observe
that by the triangle inequality, if $x\in \X_c$ for some $c>2$ then
\begin{equation*}
\begin{split}
  \left(\limsup_{N\to\infty} \frac{1}{N}\sum_{n=1}^N (x_n + \xi_n)^2\right)^{1/2} &\ge
  \left(\lim_{N\to\infty} \frac{1}{N}\sum_{n=1}^N x_n^2\right)^{1/2} - \left(\lim_{N\to\infty} \frac{1}{N}\sum_{n=1}^N
  \xi_n^2\right)^{1/2} =\\
  &= c - 1 > 1
\end{split}
\end{equation*}
almost surely, where the exceptional set of probability zero does
not depend on $x$.
\end{proof}
For the last part of the theorem, we require the following lemma.
\begin{lemma}\label{uniform_coupling_lem}
For each $0\le c\le 2$, there exists a random vector $(\xi^1,\xi^2)$
satisfying:
\begin{enumerate-math}
\item Each of $\xi^1$ and $\xi^2$ is a (real-valued) standard normal
random variable.
\item $\e(\xi^1-\xi^2)^2 = c^2$.
\end{enumerate-math}
\end{lemma}
\begin{proof}
Let $\xi^1$ be a standard normal random variable and define
\begin{equation*}
\xi^2:=\begin{cases}
\hphantom{-} \xi^1, & |\xi^1|\ge t\\  
-\xi^1, & |\xi^1|< t
\end{cases}
\end{equation*}
where $0\le t\le \infty$. It is straightforward to check that
$\xi^2$ is a standard normal random variable, and that for an
appropriate $t = t(c)$ we may achieve $\e(\xi^1-\xi^2)^2=c^2$.
\end{proof}
\begin{proof}[Proof of part \ref{item:uniform_detection_impossible} of
Theorem~\ref{thm:uniform_detection}] Fix $0<c\le 2$. Let
$T:\R^\N\to\{0,1\}$ be a Borel measurable mapping and assume, in
order to obtain a contradiction, that $\X_c$ admits uniform
detection via $T$. Consider independent copies $(\xi_n^1,\xi_n^2)$,
$n\in\N$, of the random vector of Lemma~\ref{uniform_coupling_lem}
for this $c$. By our assumption, we have almost surely that
\begin{equation*}
  T(\xi^1)=0\text{ and }T(x+\xi^2)=1\text{ for all }x\in\X_c.
\end{equation*} However,
$\xi^1-\xi^2\in\X_c$ almost surely by the law of large numbers. This
is a contradiction, since letting $x=\xi^1-\xi^2$ we see that
$T(x+\xi^2)=T(\xi^1)=0$. Since $T$ is arbitrary, this finishes the
proof of the theorem.
\end{proof}

\subsection{Partial Recovery}
When a space $\X$ does not admit recovery, one may ask instead for a
weaker property, that there exists a mapping $T$ taking the noisy
signal to a signal which is ``close'' to the original transmitted
signal. In this section we consider a rather weak notion of
``closeness'', that the recovered signal be close to the transmitted
signal in ``mean energy''. We show that even for this weak notion,
recovery is not always possible.

We again focus on the real-valued case and fix the index set of the
signals to be $\N$. Define a ``distance'' between two signals
$x,y\in\R^\N$ by
\begin{equation*}
d(x,y) := \left(\limsup_{N\to\infty} \frac{1}{N} \sum_{n=1}^N
|x_n-y_n|^2\right)^{1/2}
\end{equation*}
(this is sometimes called the Besicovitch distance, see
\cite{besicovitch}). This ``distance'' satisfies the triangle
inequality, but $d(x,y)=0$ does not imply $x=y$. In particular,
$d(x,y)=0$ whenever $x_n-y_n$ tends to zero with $n$. We will refer
to $d(x,y)$ as the \emph{mean energy distance} between $x$ and $y$.
\begin{definition*} We say that a space $\X\subseteq\R^\N$ admits \emph{recovery up to mean energy $c$ $(c\ge
0)$} if there exists a Borel measurable function $T:\R^\N\to\R^\N$
such that for each $x\in\X$, $d(T(x+\xi), x)\le c$ almost surely.
\end{definition*}

Observe that, by the strong law of large numbers, for every signal
$x\in\R^I$, $d(x+\xi, x) = 1$ almost surely. Thus recovery up to
mean energy $c = 1$ is always possible, by taking $T$ to be the
identity mapping. The following theorem shows that there exist
signal spaces for which this cannot be significantly improved.

\begin{theorem}
\label{no_linfty_rec} Let $\X=\{-1,1\}^\N$. There is a positive
constant $c$ such that $\X$ does not admit recovery up to mean
energy $c$.
\end{theorem}

We will need the existence of the following coupling.

\begin{lemma}\label{coupling_lem}
There exists a random vector $(Y^1,Y^2,\xi^1,\xi^2)$ satisfying:
\begin{enumerate-math}
\item Each of $Y^1$ and $Y^2$ is uniformly distributed on
$\{-1,1\}$.
\item Each of $\xi^1$ and $\xi^2$ is a (real-valued) standard normal random variable.
\item $Y^1$ is independent of $\xi^1$ and $Y^2$ is independent of
$\xi^2$.
\item $\P(Y^1+\xi^1=Y^2+\xi^2)=1$ and $\P(Y^1 = Y^2) < 1$.
\end{enumerate-math}
\end{lemma}

\begin{proof}
Let $\phi: \R \to \R$ be the density of a standard Gaussian random
variable. For a sufficiently small $0<p<1$, we may express $\phi$ as
a convex combination of the form
\begin{equation}\label{convex_comb_eq}
\phi(x):=p\cdot\frac{1}{4}\1_{(-2,2)}(x) + (1-p) f(x),
\end{equation}
where $\1_{(-2,2)}$ is the indicator function of the $(-2,2)$
interval, and $f$ is non-negative with integral one. Let
$Z_1,Z_2,Z_3,U, W$ and $I$ be independent random variables, where
each of $Z_1,Z_2,Z_3$ is uniformly distributed on $\{-1,1\}$, $U$ is
distributed uniformly on the segment $[-1,1]$, $W$ is distributed
with density $f$ and $I$ takes the values $0$ and $1$ with
$\P(I=1)=p$. Define
\begin{align*}
Y^1 &:= IZ_1+(1-I)Z_3, &Y^2&:=IZ_2+(1-I)Z_3,\\
\xi^1 &:= I(Z_2+U)+(1-I)W, &\xi^2&:=I(Z_1+U)+(1-I)W.
\end{align*}
It is straightforward to check that $(Y^1,Y^2,\xi^1,\xi^2)$
satisfies the requirements of the lemma. For instance, to check that
$\xi^1$ is a standard normal random variable, note that $Z_2+U$ is
distributed uniformly on the segment $[-2,2]$ and apply
\eqref{convex_comb_eq}. To check that $Y^1$ is independent of
$\xi^1$, note that for any two Borel sets $A,B\subseteq\R$ we have
\begin{equation*}
\begin{split}
\P(Y^1\in A,\, \xi^1\in B) &= \e\left[\P(Y^1\in A,\, \xi^1\in B
\,|\,I) \right]
 = \e\left[\P(Y^1\in A\, |\, I)\,\P(\xi^1\in B\, |\, I)\right]\\
&= \P(Y^1\in A)\, \e\left[\P(\xi^1\in B\, |\, I)\right]
 = \P(Y^1\in A) \, \P(\xi^1\in B),
\end{split}
\end{equation*}
where we have used the fact that $\P(Y^1\in A\, |\, I)$ is a
constant.
\end{proof}

\begin{proof}[Proof of Theorem~\ref{no_linfty_rec}]
Let $T:\R^\N\to\R^\N$ be a Borel measurable function. Consider
independent copies $(Y^1_n,Y^2_n,\xi^1_n,\xi^2_n)$, $n\in\N$, of the
random vector of Lemma~\ref{coupling_lem}. Denote
$\bar{Y}^1:=(Y^1_n), \bar{Y}^2:=(Y^2_n), \bar{\xi}^1_n:=(\xi^1_n)$
and $\bar{\xi}^2_n:=(\xi^2_n)$. By the triangle inequality and the
fact that $\bar{Y}^1 + \bar{\xi}^1 = \bar{Y}^2 + \bar{\xi}^2$ almost
surely, we have
\[
d(\bar{Y}^1,\bar{Y}^2) \le d(T(\bar{Y}^1+\bar{\xi}^1),\bar{Y}^1) +
d(T(\bar{Y}^2+\bar{\xi}^2),\bar{Y}^2)\quad\text{almost surely}.
\]
Observe that, by the strong law of large numbers, the distance
$d(\bar{Y}^1,\bar{Y}^2)$ is almost surely constant, and is equal to
\[
\sigma := \sqrt{ \e \, [Y^1_1-Y^2_1]^2 } > 0.
\]
It follows that there is a $j\in\{1,2\}$ such that, with positive
probability, $d(T(\bar{Y}^j+\bar{\xi}^j),\bar{Y}^j)\ge \sigma/2$.
Since $\bar{Y}^j\in \X$ almost surely, and $\bar{Y}^j$ is
independent of $\bar{\xi}^j$, it follows from Fubini's theorem that
there exists some $y\in \X$ for which
$d(T(y+\bar{\xi}^j),y)\ge\sigma/2$ with positive probability. Since
$T$ is arbitrary, this proves the theorem.
\end{proof}


\section{Remarks and open
questions}\label{sec:remarks_open_questions}

{\bf Necessary and sufficient conditions.} Are there useful
necessary and sufficient conditions for a space $\X$ to admit
detection or recovery? The only condition of this kind that we have
is the detection criterion given by
Theorem~\ref{thm:detection_singularity_criterion}. However, this
criterion does not seem simple to check in concrete examples.

{\bf Quantitative recovery and detection.} In this work we discussed
the notion of almost-sure recovery (or detection) from infinite
noisy signals. The assumption that the signal has infinitely many
coordinates is certainly necessary for recovery to be possible with
probability one. In a more quantitative setup one may consider
signals having only finitely many coordinates and ask that recovery
be attained with some probability $p\in(0,1)$. It is of interest to
find such quantitative analogs of our results. One may try, for
instance, to take a space $\X$ for which almost-sure recovery is
possible and create from it a sequence of spaces $\X_n$, with
signals in $\X_n$ having only $n$ coordinates, such that $\X_n$
tends to $\X$ in some sense as $n$ tends to infinity. Then one may
try to investigate the probability of recovery from the spaces
$\X_n$ when $n$ is large.

It is not clear what the appropriate way to define $\X_n$ should be.
To give an example, let $\X$ be the space of all periodic signals in
$\R^\N$, which certainly admits recovery. One may naively define
$\X_n$ to be the space of all signals in $\R^n$ obtained as the
first $n$ coordinates of a signal in $\X$. However, with this
definition $\X_n = \R^n$ and no useful recovery is possible. To
remedy this, one may choose a function $f:\N\to\N$ which tends to
infinity and limit $\X_n$ to the set of vectors of the first $n$
coordinates of periodic signals with period at most $f(n)$. Such a
definition, although useful, is not unique as it depends on the
choice of $f$ and it is not clear how to generalize it for other
spaces $\X$.

A related subject for investigation is the complexity of recovery
from the spaces $\X_n$, that is, how many operations are required in
a recovery algorithm for $\X_n$. The analogous questions for the
detection problem are also of interest.

{\bf Critical phenomena.} Several of our results are of the
following type. If a certain parameter $\sigma(X)$ of the signal
space $\X$ exceeds a bound $\sigma_1$ then $\X$ admits recovery
while if $\sigma(X)$ is smaller than another bound $\sigma_0$ then
$\X$ does not even admit detection. It is of interest to investigate
further the critical or near-critical cases, when $\sigma_0\le
\sigma(X)\le \sigma_1$.

One example is furnished by
Theorem~\ref{thm:reconstruction-small-dimension}. Let $\F_\alpha$
consist of the Fourier transforms of all finite, complex measures
$\mu$, such that $\mu$ is carried by a Borel set of Hausdorff
dimension $\le \alpha$. Does $\F_\alpha$ admit recovery or detection
when $\alpha=1/2$? Is there a finer Hausdorff gauge function (that
is, refining the Hausdorff dimension) which captures the critical
phenomenon better?

Similarly, let $\X$ be a Walsh space with non-increasing amplitudes
$\{|\sigma_n|\}$ as in part (ii) of
Theorem~\ref{thm:detection_phases}. Denote by $\sigma(\X)$ the
quantity in \eqref{eq:cond_non_detect_Walsh}. Is there a critical
threshold $\sigma_c$, independent of $\X$, such that $\X$ admits
recovery when $\sigma(\X)>\sigma_c$ and does not even admit
detection when $\sigma(\X)<\sigma_c$? If so, what happens when
$\sigma(\X)=\sigma_c$?

{\bf Polynomial phase functions.} Does the space of polynomial phase
functions admit recovery? Some indication that the answer is
positive is provided in Section~\ref{sec:poly_phase}.

{\bf Monotonicity in strength of noise.} It was shown in
Proposition~\ref{prop_monotonicity_in_noise} that if $\X$ admits
detection (or recovery) then so does $c\X$ for any $c\ge 1$. Let
\begin{equation*}
\bar{\X}:=\{cx\colon x\in\X, c\ge 1\}.
\end{equation*}
Is it also the case that if $\X$ admits detection then so does
$\bar{\X}$? One difficulty here is that $\bar{\X}$ is an
\emph{uncountable} union of spaces of the form $c\X$ so that we
cannot use Proposition~\ref{prop:union_of_spaces}. On the one hand,
detection from $\bar{\X}$ may not be harder than from $\X$ since the
signals in $\bar{\X}$ are ``amplified''. On the other hand,
detection may be more difficult due to the fact that the
``amplification factor'' is unknown to the receiver.

Similarly, one may ask for a space $\X$ admitting detection whether
\begin{equation*}
\tilde{\X}:=\{\{c_n x_n\}\colon x\in\X, c_n\ge 1\}
\end{equation*}
also admits detection. Observe that $\tilde{\X}$ does not admit
recovery since the necessary condition \eqref{eq:recover_nec_cond}
is violated.

{\bf Stronger necessary conditions.} Can one strengthen the
necessary condition for detection given by
\eqref{eq:detection_nec_cond} in any way? For instance, is it the
case that any space $\X\subseteq\R^\N$ admitting detection possesses
a decomposition $\X = \cup \X_j$ satisfying
\begin{equation*}
  \text{for each $j$, }\lim_{k\to\infty}\inf_{y\in\X_j} \sum_{n=1}^k y_n^2=\infty?
\end{equation*}
A similar question may be asked for the necessary condition for
recovery given by \eqref{eq:recover_nec_cond}.

{\bf Acknowledgement.} We are grateful to Boris Tsirelson for
showing us the proof of
Theorem~\ref{thm:detection_singularity_criterion}. We thank Yoav
Benjamini and Felix Abramovitch for useful references and thank an
anonymous referee for useful comments on the exposition.


\end{document}